\documentclass[11pt]{article}

\usepackage{amsmath,amssymb,amsfonts,amsthm}
\usepackage{mathrsfs,bm}
\usepackage{graphicx,float}
\usepackage{xcolor}
\usepackage[numbers,sort&compress]{natbib}
\usepackage{enumitem}
\usepackage{hyperref}

\allowdisplaybreaks

\hypersetup{
    colorlinks=true,
    linkcolor=blue,
    citecolor=red,
    urlcolor=blue
}

\newcommand{\hK}[1]{\|#1\|_K}
\newcommand{\RN}{\mathbb R^N}
\newcommand{\Ksc}{\mathcal K_{(o)}^{ss}}

\newtheorem{theorem}{Theorem}[section]
\newtheorem{lemma}[theorem]{Lemma}
\newtheorem{remark}[theorem]{Remark}
\newtheorem{proposition}[theorem]{Proposition}

\numberwithin{equation}{section}

\title{Sharp $L^2$ stability for anisotropic radial Caffarelli-Kohn-Nirenberg inequalities}
\author{Zhenzhen Wei, Denghui Wu$^{*}$}
\date{}

\begin{document}

\maketitle

\begingroup
\makeatletter
\renewcommand{\thefootnote}{}
\renewcommand{\@makefntext}[1]{\noindent#1}
\makeatother

\footnotetext{


$^{*}$\textit{Corresponding author.}
}

\endgroup
\begin{abstract}
We establish quantitative stability for radial Caffarelli-Kohn-Nirenberg inequalities involving the anisotropic radial derivative and the Minkowski functional of a smooth, strictly convex body in $\mathbb R^N$, $N\ge2$. Here, the convex body need not be origin-symmetric. For the $L^2$-Caffarelli-Kohn-Nirenberg inequality, we establish the sharp stability constant $|b+1-a|$ for $a\ne b+1$, and classify all equality cases in the natural weighted completion spaces, by optimizing over the angular coefficients and a common dilation parameter. We also obtain the sharp stability constants $2$ and $4$ for the scale-invariant and scale-noninvariant $L^2$-Heisenberg inequality, respectively. For $2\le p<\infty$, we also establish weighted stability of the anisotropic radial $L^p$-Caffarelli-Kohn-Nirenberg inequalities when $a\neq b+1$ and $(b+1-a)(N-p(b+1))\geq0$, and the anisotropic radial $L^p$-Heisenberg inequalities. As a byproduct, we establish weighted $L^p$-Poincar\'e inequalities for Gamma measures, which serve as the main tool for our stability estimates.
\end{abstract}

\section{Introduction}\label{s1}

This paper develops a quantitative stability theory for anisotropic
Caffarelli-Kohn-Nirenberg (CKN) inequalities and Heisenberg uncertainty
principles associated with the Minkowski functional of a convex body.
See, e.g., \cite{CFL,CFLL,WW22,WW24}
for stability results concerning the classical CKN inequalities, and
\cite{Wei1,Wei2} for the anisotropic CKN inequalities considered here.

The CKN inequalities, introduced by Caffarelli, Kohn and Nirenberg
\cite{CKN}, form a fundamental class of weighted interpolation
inequalities in Euclidean spaces.
These inequalities include the Hardy inequality and
the Heisenberg uncertainty principle as special cases.
Their sharp constants, extremal functions and symmetry properties
have been studied extensively
(see, e.g., \cite{CW,FS,Lam19,Lam21,LL,LewisLiLi2012}).
Related sharp weighted Trudinger-Moser and CKN inequalities
were investigated in \cite{DgLL}.
Related identities, remainder terms and improvements
can be found in \cite{ACP,CFLL,DFLL}.

Once the sharp inequality is known, a natural question is whether
its deficit controls the distance to the extremal family.
The work of Bianchi and Egnell \cite{BE} established
a quantitative stability estimate for the Sobolev inequality. Related Sobolev inequalities with remainder terms
were studied in \cite{BL,LW00}.
For subsequent stability results concerning Sobolev type
inequalities and CKN inequalities, see, e.g.,
\cite{CLT24,CFL,CLT242,CLT243,CLTW,DLLN,WW22}.

In this paper, we work on the Euclidean space $\mathbb R^N$,
$N\ge2$, and consider real-valued functions.
We use $x\cdot y$ for the Euclidean
inner product and $C_c^\infty(\Omega)$ for the smooth functions with
compact support in an open set $\Omega$.
A convex body is a compact convex set with nonempty interior.
Let $\Ksc$ denote the class of smooth, strictly convex bodies in $\mathbb R^N$ containing the origin $o$ in their interior.
For a fixed $K\in\Ksc$, we denote by $\partial K$ its boundary.
The Minkowski functional of $K$ is
\[
\|x\|_K=\inf\{t>0:x\in tK\},\quad x\in\mathbb R^N,
\]
where $tK=\{tz:z\in K\}$.
We do not assume that $K$ is origin-symmetric. Its
Minkowski functional need not be even, and need not be a norm.
In this framework, Figalli, Maggi and Pratelli \cite{FMP10,FMP}
proved quantitative versions of the anisotropic
isoperimetric and Sobolev inequalities.
See \cite{DPDG,DLL,LLZ19,LLZ20,LiYan23,LSXZ,RS,RSY} for more information about
anisotropic isoperimetric, Hardy, and CKN inequalities
in the setting of norms and quasi-norms.

This paper studies the stability of the anisotropic radial CKN
inequalities associated with Minkowski functionals.
The underlying inequalities and their extremals were obtained
in \cite{Wei1,Wei2}.
Do, Lam, Lu and Nguyen \cite{DLLN} established sharp $L^2$
stability and $L^p$ extensions for full-gradient CKN inequalities.
In the present radial derivative setting, the energy does not control
angular variation, and the extremal family allows arbitrary angular coefficients.
We characterize all equality cases by combining radial spectral analysis with minimization over the common parameter.

For $x\ne o$, we set $r=\|x\|_K,$
$\sigma=\sigma_K(x)=\frac{x}{\|x\|_K}\in\partial K$,
and $\mathcal R_K(u)(x)=\frac{x\cdot\nabla u(x)}{\|x\|_K}.$
In anisotropic polar coordinates, one has
$\mathcal R_K(u)(r,\sigma)=\partial_r u (r,\sigma)$. Let $\mu_K$ be the finite Borel measure on $\partial K$ characterized by
\[
\int_{\mathbb R^N}f(x)dx
=\int_{\partial K}\int_0^\infty
f(r\sigma)r^{N-1}drd\mu_K(\sigma)
\]
for every nonnegative measurable function $f$.
For $1<p<\infty$ and $a,b\in\mathbb R$, let $\mathcal C^p_{K,a,b}$ be the completion of $C_c^\infty(\mathbb R^N\setminus\{o\})$  with respect to
\begin{equation}\label{eq:completion-norm}
\|u\|_{\mathcal C^p_{K,a,b}}^p
=\int_{\mathbb R^N}\bigl|\mathcal R_K(u)\bigr|^p r^{-bp} dx
+\int_{\mathbb R^N}|u|^p r^{-ap} dx.
\end{equation}
Every element of $\mathcal C^p_{K,a,b}$ has a unique
representative in
$L^p_{\rm loc}(\mathbb R^N\setminus\{o\})$.
For $\mu_K$-almost every $\sigma\in\partial K$,
the function $r\mapsto u(r\sigma)$ is locally
absolutely continuous on $(0,\infty)$, with derivative
$\mathcal R_K(u)(r\sigma)$
(see \cite[Lemma~4.1]{Wei2}).
For $p=2$, we define
\begin{align*}
I_a(u)=\int_{\mathbb R^N}|u|^2r^{-2a} dx,\ \
D_b(u)=\int_{\mathbb R^N}\bigl|\mathcal R_K(u)\bigr|^2r^{-2b} dx,\ \
I_{a,b}(u)=\int_{\mathbb R^N}|u|^2r^{-a-b-1} dx.
\end{align*}
If $b+1-a\neq0$, we set
\begin{equation}\label{eq:unified-parameters}
\gamma=\frac{\operatorname{sgn}(b+1-a)|N-2b-2|-\bigl(N-2b-2\bigr)}{2},\quad
\nu=1+\frac{|N-2b-2|}{|b+1-a|}\ge1.
\end{equation}
The sharp radial CKN inequality was proved in \cite{Wei1}
with constant
$\frac{\nu|b+1-a|}{2}=\frac{|N-2b-2|+|b+1-a|}{2}$.
Hence $\mathcal C^2_{K,a,b}$ embeds continuously into
$L^2(\RN,r^{-a-b-1}dx)$.
The extremal family is
\begin{equation}\label{eq:unified-profile}
\mathcal M_{K,a,b}
=\{\Phi(\sigma_K)g_\eta(\|x\|_K):\eta>0,\ \Phi\in L^2(\partial K,d\mu_K)\},
\end{equation}
where $g_\eta(r)=r^\gamma e^{-\frac{\eta r^{b+1-a}}{|b+1-a|}}$ for $\eta>0$.
For $u\neq 0$, we define
\begin{gather}
\delta_{K,a,b}(u)=\sqrt{D_b(u)I_a(u)}-\frac{\nu|b+1-a|}{2}I_{a,b}(u),
\label{eq:unified-deficit}\\
d_{K,a,b}(u,\mathcal M_{K,a,b})
= \inf_{v\in\mathcal M_{K,a,b}}I_{a,b}(u-v)
^{\frac{1}{2}}.
\nonumber
\end{gather}
For angular coefficients, we write
$\langle\Phi,\Psi\rangle=\int_{\partial K}\Phi\Psi d\mu_K,$ and
$\|\Phi\|_2^2=\int_{\partial K}\bigl|\Phi\bigr|^2d\mu_K.$

Our first main result is the sharp $L^2$ stability
for the anisotropic radial CKN inequality in
the full range $a\ne b+1$, with the characterization of equality cases.

\begin{theorem}\label{thm:unified}
Let $K\in\Ksc$ and $a\ne b+1$.
For every nonzero $u\in\mathcal C^2_{K,a,b}$,
\begin{equation}\label{eq:unified-stability}
\delta_{K,a,b}(u)\ge |b+1-a|\,d_{K,a,b}(u,\mathcal M_{K,a,b})^2.
\end{equation}
The constant $|b+1-a|$ is sharp, including on $C_c^\infty(\mathbb R^N\setminus\{o\})$.
Equality holds if and only if
\begin{equation}\label{eq:unified-equality}
u(r,\sigma)=g_\lambda(r)\bigl[\Phi_0(\sigma)+\Phi_1(\sigma)\bigl(s-\nu\bigr)\bigr],
\quad s=\frac{2\lambda r^{b+1-a}}{|b+1-a|},
\end{equation}
for $\lambda=\lambda_2(u):=\sqrt{\frac{D_b(u)}{I_a(u)}}>0$,
and $\Phi_0,\Phi_1\in L^2(\partial K,d\mu_K)$ not both zero
and satisfying
\begin{equation}\label{eq:unified-conditions}
\langle\Phi_0,\Phi_1\rangle=0,\quad
\|\Phi_0\|_2^2\ge\nu\|\Phi_1\|_2^2.
\end{equation}
No nonzero compactly supported smooth function attains equality.
\end{theorem}

For the functions in \eqref{eq:unified-equality}, the orthogonality condition
in \eqref{eq:unified-conditions} is equivalent to the balance relation
$D_b(u)=\lambda^2 I_a(u)$, while the norm condition ensures that
$\lambda$ is a global nearest scale.
More precisely, Proposition~\ref{prop:nearest-scales} shows that, within
the class of linear combinations of the first two radial Laguerre modes with orthogonal angular coefficients,
the nearest scale is unique when $\|\Phi_0\|_2^2\ge\nu\|\Phi_1\|_2^2$,
whereas there are exactly two nearest scales when
$\|\Phi_0\|_2^2<\nu\|\Phi_1\|_2^2$.
Therefore, the equality cases in Theorem~\ref{thm:unified}  include
functions with positive distance from $\mathcal M_{K,a,b}$ whenever
$\Phi_1\ne0$ and \eqref{eq:unified-conditions} holds.

The Heisenberg uncertainty principle is a special case
of the CKN inequalities.
Sharp uncertainty principles have been investigated
using factorization methods in \cite{CFLso,KLLSW25}.
Their quantitative stability has also attracted attention,
see \cite{Fathi,MV} for the classical Heisenberg inequality,
and \cite{DLLH} for more general cases.

In our anisotropic radial setting, taking $(a,b)=(-1,0)$ gives
$b+1-a=2,$ $N-2b-2=N-2\ge0,$ $\gamma=0,$ $\nu=\frac N2.$
Let
\[
\mathcal H_K=\mathcal C^2_{K,-1,0},\quad
g_\lambda(r)=e^{-\frac{\lambda r^2}{2}},\quad
\mathcal M_{K,H}=\mathcal M_{K,-1,0},\quad
\delta_{K,H}(u)=\delta_{K,-1,0}(u).
\]
Applying Theorem \ref{thm:unified}, we show
scale-invariant $L^2$ Heisenberg stability in Theorem \ref{thm:H-L2}:
For every nonzero $u\in\mathcal H_K$,
\begin{equation*}
\delta_{K,H}(u)
\geq 2 \inf_{v\in\mathcal M_{K,H}}\|u-v\|^2_{L^2(\mathbb R^N)}.
\end{equation*}
Theorem \ref{thm:H-L2} also establishes a scale-noninvariant $L^2$ Heisenberg stability, which is:
For every $u\in\mathcal H_K$,
\begin{equation*}
\delta_{K,H,non}(u)
\geq
4 \inf_{v\in\mathcal M_{K,H,non}}\|u-v\|^2_{L^2(\mathbb R^N)},
\end{equation*}
where $\delta_{K,H,non}(u)
=\|\mathcal R_K(u)\|^2_{L^2(\mathbb R^N)}
+\|ru\|^2_{L^2(\mathbb R^N)}
-N\|u\|^2_{L^2(\mathbb R^N)}$ and
$\mathcal M_{K,H,non}
=\big\{\Phi(\sigma)e^{-\frac{r^2}{2}}:\Phi\in L^2(\partial K,d\mu_K)\big\}.$
The constants 2 and 4 are sharp, including on $C_c^\infty(\RN\setminus\{o\})$.

For $1<p<\infty$, define the parameter region
$$
\mathcal A_p=
\Big\{(a,b)\in\mathbb R^2:b+1-a>0, b \leq \frac{N-p}{p}\Big\}
\cup
\Big\{(a,b)\in\mathbb R^2:b+1-a<0, b \ge\frac{N-p}{p}\Big\}.
$$
For $(a,b)\in\mathcal A_p$,
$$\nu_p=\frac{N-1-(p-1)a-b}{b+1-a}.$$
Then
$
\nu_p
=p-1+\frac{N-p(b+1)}{b+1-a}
\ge p-1>0.
$
In particular, $\nu_p\ge1$ when $p\ge2$.
If $(a,b)\in\mathcal A_2$, then
$\gamma=0$ in \eqref{eq:unified-parameters}
and the $L^2$-extremal family is
$$\mathcal M_{K,a,b}
=\Big\{\Phi(\sigma_K(x))\exp\Bigl(-\frac{\eta}{|b+1-a|}\|x\|_K^{b+1-a}\Bigr):
\eta>0,\ \Phi\in L^2(\partial K,d\mu_K)\Big\}.$$
Generally, the extremal family of the anisotropic $L^p$-CKN inequality
in $\mathcal A_p$ was obtained, in \cite{Wei2}, by
$$\mathcal M_{p,K,\mathcal A_p}
=\Big\{\Phi(\sigma_K(x))\exp\Big(-\frac{\eta}{|b+1-a|}\|x\|_K^{b+1-a}\Big):
\eta>0,\
\Phi\in L^p(\partial K,d\mu_K)\Big\}.$$
For $u\in\mathcal C^p_{K,a,b}$,
we define the deficit by
\begin{align*}
\delta_{p,K,\mathcal A_p}(u)
\!=\!\biggl(\int_{\mathbb R^N}\!\frac{\bigl|\mathcal R_K(u)\bigr|^p}{r^{bp}} dx\biggr)^{\frac{1}{p}}
\biggl(\int_{\mathbb R^N}\!\frac{|u|^p}{r^{ap}} dx\biggr)^{\frac{1}{p'}}\!-\!\frac{|N\!-\!1\!-\!(p\!-\!1)a\!-\!b|}{p}
\int_{\mathbb R^N}\!\frac{|u|^p}{r^{(p-1)a+b+1}} dx,
\end{align*}
where $p'=\frac{p}{p-1}$, and the weighted $L^p$ distance by
\begin{align*}
d_{p,K,\mathcal A_p}
\bigl(u,\mathcal M_{p,K,\mathcal A_p}\bigr)^p
=\inf_{v\in\mathcal M_{p,K,\mathcal A_p}}
\int_{\mathbb R^N}
\frac{|u-v|^p}
{\|x\|_K^{(p-1)a+b+1}}dx.
\end{align*}
The deficit and the weighted distance are finite
by Lemma~\ref{lApembedding}.

The following theorem establishes weighted $L^p$ stability
for the anisotropic radial CKN inequality
in the region $\mathcal A_p$. The constant is not claimed to be sharp when $p>2$.

\begin{theorem}\label{tA}
Let $2\le p<\infty$, $K\in\Ksc$ and
$(a,b)\in\mathcal A_p$. Then, for every
$u\in \mathcal C_{K,a,b}^p\setminus\{0\},$
one has
\begin{align*}
\delta_{p,K,\mathcal A_p}(u)
\geq \frac{\kappa_p p^{p-2}|b+1-a|}{P_{\nu_p,p}}d_{p,K,\mathcal A_p}
\bigl(u,\mathcal M_{p,K,\mathcal A_p}\bigr)^p,
\end{align*}
where $\kappa_p$ and $P_{\nu_p,p}$ are given in Lemmas~\ref{lUC} and
\ref{l2} respectively.
In particular, $\kappa_2=P_{\nu_2,2}=1$.
\end{theorem}

To describe the $L^p$ Heisenberg estimates, let $2\le p<\infty$, $p'=\frac{p}{p-1}$ and
$\nu_H=\frac{N}{p'}$. We write
$\mathcal H_K^p=\mathcal C^p_{K,-\frac{1}{p-1},0}.$
The space $\mathcal H_K^p$ embeds continuously into
$L^p(\RN)$, as shown in the proof of
Theorem~\ref{tHcompletion}.
For $u\in\mathcal H_K^p$, the scale-noninvariant
and scale-invariant deficits are defined by
\begin{gather*}
\delta_{p,K,H,non}(u)
=\int_{\RN}\bigl|\mathcal R_K(u)\bigr|^pdx
+(p-1)\int_{\RN}r^{p'}|u|^pdx
-N\int_{\RN}|u|^pdx,\\
\delta_{p,K,H}(u)
=\biggl(\int_{\RN}\bigl|\mathcal R_K(u)\bigr|^pdx\biggr)^{\frac{1}{p}}
\biggl(\int_{\RN}r^{p'}|u|^pdx\biggr)^{\frac{1}{p'}}
-\frac{N}{p}\int_{\RN}|u|^pdx.
\end{gather*}
For the stability estimates, we introduce the following comparison families
\begin{gather*}
\mathcal M_{p,K,H,non}=\Big\{\Phi(\sigma)e^{-\frac{r^{p'}}{p'}}:\Phi\in L^p(\partial K,d\mu_K)\Big\},\\
\mathcal M_{p,K,H}=\Big\{\Phi(\sigma)e^{-\frac{\eta r^{p'}}{p'}}:\eta>0,\ \Phi\in L^p(\partial K,d\mu_K)\Big\},
\end{gather*}
When $p>N$, the nonzero elements of these families do not belong to $\mathcal H_K^p$. They nevertheless belong to $L^p(\mathbb R^N)$
and serve as comparison functions in the stability below, see Remark~\ref{rem:trace}.

The following theorem gives the scale-invariant and scale-noninvariant $L^p$ Heisenberg stability for the full range $2\leq p<\infty$.

\begin{theorem}\label{tHcompletion}
Let $2\le p<\infty$ and $K\in\Ksc$.
\begin{enumerate}[label=\textup{(\roman*)}]
\item
For every $u\in\mathcal H_K^p$,
\begin{align*}
\delta_{p,K,H,non}(u)
\geq
\frac{\kappa_p p'p^{p-1}}{P_{\nu_H,p}}
\inf_{v\in\mathcal M_{p,K,H,non}}\|u-v\|^p_{L^p(\mathbb R^N)}.
\end{align*}
\item
For every $u\in\mathcal H_K^p\setminus\{0\},$
\begin{align*}
\delta_{p,K,H}(u)
\geq
\frac{\kappa_p p'p^{p-2}}{P_{\nu_H,p}}
\inf_{v\in\mathcal M_{p,K,H}}\|u-v\|^p_{L^p(\mathbb R^N)}.
\end{align*}
\end{enumerate}
\end{theorem}

Note that in the case $p=2$, the above $L^p$ Heisenberg stability is sharp
as shown in Theorem \ref{thm:H-L2}.
As the cases $2\le p\le N$ are in the range $\mathcal A_p$,
Theorem~\ref{tHcompletion} \textup{(ii)} follows
from Theorem~\ref{tA} with $a=-\frac{1}{p-1}$ and $b=0$.
For $p>N$, the parameter pair $(a,b)$ in Theorem~\ref{tHcompletion} lies outside $A_p$.

The main tools in the proofs are radial remainder identities and weighted $L^p$-Poincar\'e inequalities for Gamma measures.
As a technical result, we use a weighted Hardy criterion \cite{KPS,Muckenhoupt} to
establish the latter inequalities on $C_c^\infty(0,\infty)$.
For $p=2$, we extend the inequality to the natural energy space
$\mathcal D_\nu=\big\{
h\in W_{\mathrm{loc}}^{1,2}(0,\infty):\int_0^\infty
\big(|h(s)|^2+s|h'(s)|^2\big)
s^{\nu-1}e^{-s}ds<\infty\big\},$ and characterize the sharp constant and all
equality cases using generalized Laguerre polynomials and
estimates of the boundary terms.

Section~\ref{s2} contains the geometric notation
and auxiliary inequalities.
Section~\ref{s2-1} establishes the weighted
$L^p$-Poincar\'e inequalities for Gamma measures.
Section~\ref{s-sharp} proves Theorem~\ref{thm:unified}
and establishes the sharp $L^2$ Heisenberg stability estimates
in Theorem~\ref{thm:H-L2}.
Section~\ref{s-lp} establishes the general $L^p$
stability estimates, first for
compactly supported smooth functions and then for the functions in the related completion spaces.

\section{Preliminaries}\label{s2}

We recall the geometric notation and polar integration formula,
and record a uniform convexity estimate and a lemma on
minimization over angular coefficients.

Let $N\ge2$ and let $\RN$ be the $N$-dimensional Euclidean space.
We write $o$ for the origin, $x\cdot y$ for the Euclidean inner
product, and $|x|=(x\cdot x)^{\frac12}$ for the Euclidean norm. We use the
standard notation $C_c^\infty(\Omega)$ and $L^p(\Omega)$ for smooth
compactly supported functions and Lebesgue spaces, respectively.

A convex body is a compact convex subset of $\RN$ with nonempty interior.
We denote by $\mathcal K_{(o)}$ the class of convex bodies containing the origin in their interiors. For a convex body $K$, its boundary is denoted by $\partial K$.

For $K\in\mathcal K_{(o)}$, its Minkowski functional is defined by
$$
\hK{x}
=
\inf\{\lambda>0:x\in \lambda K\},
\quad x\in\RN.
$$
By this definition, we have
$$
K=\{x\in\RN:\hK{x}\le 1\},
\quad
\partial K=\{x\in\RN:\hK{x}=1\}.
$$
Moreover, $\|\cdot\|_K$ is positively homogeneous, namely
$
\|tx\|_K=t\|x\|_K$ for $t>0$ and $x\in\RN.
$
If $K$ is origin-symmetric, then $\|\cdot\|_K$ is a norm. In general, if
$K$ is not origin-symmetric, $\|\cdot\|_K$ need not be even and hence need
not be a norm.

A convex body $K\in\mathcal K_{(o)}$ is said to be smooth if each boundary
point of $K$ has a unique outer unit normal vector. It is said to be strictly
convex if $\partial K$ does not contain any nontrivial line segment. We denote
by
$
\Ksc
$
the class of convex bodies in $\mathcal K_{(o)}$ which are smooth and strictly
convex.

For $K\in\mathcal K_{(o)}$ and $x\ne o$, let
$
r_K=\hK{x}$ and
$
\sigma_K(x)=\frac{x}{\hK{x}}.
$
Then $\sigma_K(x)\in\partial K$ and
$x=r_K\sigma_K(x).$
For simplicity, we omit the subscript $K$ and write
$x\in\RN\setminus\{o\}$ in the form
$
x=r\sigma$ for $r>0$ and $\sigma\in\partial K.$
Let $d\mu_K$ be the finite Borel measure on $\partial K$ induced by the
anisotropic polar formula
$$
\int_{\RN}u(x)dx
=
\int_0^\infty\int_{\partial K}
u(r\sigma)r^{N-1}d\mu_K(\sigma)dr,
$$
for every nonnegative measurable function $u$ on $\RN$. We denote by
$L^p(\partial K,d\mu_K)$ the corresponding $L^p$ space on $\partial K$.

For $u\in C^1(\mathbb R^N\setminus\{o\})$, the anisotropic radial
derivative with respect to $K$ is defined by
$$
\mathcal R_K(u)(x)
=
\nabla u(x)\cdot \sigma_K(x)
=
\nabla u(x)\cdot \frac{x}{\hK{x}},
\quad x\ne o.
$$
In the coordinates $x=r\sigma$, where $r=\hK{x}$ and $\sigma=\sigma_K(x)$,
one has
$\mathcal R_K(u)(x)=\partial_r u(r,\sigma).$

For $\tau\in\mathbb R\setminus\{0\}$, let
$$\operatorname{sgn}(\tau)=
\begin{cases}
1, & \tau>0,\\
-1, & \tau<0.
\end{cases}$$
For $1<p<\infty$ and $s,t\in\mathbb R$, the Bregman divergence of the convex function $t\mapsto |t|^p$ is defined by
$$
R_p(s,t)=|t|^p+(p-1)|s|^p-p|s|^{p-2}st,
$$
with the usual convention that $|s|^{p-2}s=0$ when $s=0$. In particular,
when $p=2$,
$$
R_2(s,t)=|t-s|^2.
$$

When $p>1$, we have $R_p(s,t)\geq0$ for all real $s,t$, with equality if and only if $s=t$.
We next give the standard uniform convexity estimate for the
function $t\mapsto |t|^p$ for $p\geq2$ (see, e.g., \cite[p 99]{Lindqvist}).

\begin{lemma}\label{lUC}
Let $2\le p<\infty$. Then there exists $\kappa_p>0$ such that
$$
R_p(s,t)\ge\kappa_p|t-s|^p
$$
for all $s,t\in\mathbb R$.
In particular, when $p=2$, one has
$R_2(s,t)=|t-s|^2$, and we take $\kappa_2=1$.
\end{lemma}

We next give a minimization lemma, related to one dimensional
minimization along each anisotropic direction to the corresponding
global minimization.

\begin{lemma}\label{l5}
Let $(X,\mu)$ be a finite measure space, let $1<p<\infty$, and let
$v,w:(0,\infty)\to(0,\infty)$ be measurable functions satisfying
$$
0
<
\int_0^\infty v(r)^p w(r)dr
<
\infty.
$$
Suppose that $u:X\times(0,\infty)\to\mathbb R$ is measurable and
$$
\int_X\int_0^\infty
|u(\xi,r)|^p w(r)drd\mu(\xi)
<
\infty.
$$
Then
$$
\int_X
\inf_{c\in\mathbb R}
\int_0^\infty
|u(\xi,r)-cv(r)|^p w(r)drd\mu(\xi)
=
\inf_{\Phi\in L^p(X,d\mu)}
\int_X\int_0^\infty
|u(\xi,r)-\Phi(\xi)v(r)|^p
w(r)drd\mu(\xi).
$$
\end{lemma}

\begin{proof}
By Fubini's theorem, there exists a measurable set
$X_0\subset X$ with $\mu(X\setminus X_0)=0$ such that
\[
\int_0^\infty |u(\xi,r)|^p w(r)dr<\infty
\quad\text{for every }\xi\in X_0.
\]
All pointwise constructions below are made on $X_0$,
and we set $\Phi_n=0$ on $X\setminus X_0$.

For almost every $\xi\in X_0$, the function
$$c\mapsto
\int_0^\infty
|u(\xi,r)-cv(r)|^p w(r)dr
$$
is continuous, and then measurable. Since $\mathbb Q$ is dense in $\mathbb R$, one has
$$
\inf_{c\in\mathbb R}
\int_0^\infty
|u(\xi,r)-cv(r)|^p w(r)dr
=\inf_{q\in\mathbb Q}\int_0^\infty|u(\xi,r)-qv(r)|^p w(r)dr.
$$
Therefore, the function on the right side is measurable with respect to $\xi$,
and hence so is the left side.

Let
$\mathbb Q=\{q_1,q_2,\ldots\}.$
For each $n\in\mathbb N$ and $\xi\in X_0$, we define $\Phi_n(\xi)$ to be
the first $q_k$ satisfying
\begin{align}\label{eq2.1}
\int_0^\infty
\bigl|u(\xi,r)-q_kv(r)\bigr|^p w(r)dr
\leq
\inf_{c\in\mathbb R}\int_0^\infty\bigl|u(\xi,r)-cv(r)\bigr|^p w(r)dr
+\frac1n.
\end{align}
For $n$ and $k$, the function $
\inf_{c\in\mathbb R}\int_0^\infty\bigl|u(\xi,r)-cv(r)\bigr|^p w(r)dr
-\int_0^\infty
\bigl|u(\xi,r)-q_kv(r)\bigr|^p w(r)dr+\frac1n$ is measurable. Hence the sets, defined by
$$E_{k,n}=\left\{\xi\in X_0:\int_0^\infty
\bigl|u(\xi,r)-q_kv(r)\bigr|^p w(r)dr
\leq
\inf_{c\in\mathbb R}\int_0^\infty\bigl|u(\xi,r)-cv(r)\bigr|^p w(r)dr
+\frac1n\right\},$$
are measurable.
Set $B_{1,n}=E_{1,n}$ and $B_{k,n}=E_{k,n}\backslash (\bigcup_{i=1}^{k-1}E_{i,n})$ for $k=2,3\cdots$.
These sets are measurable and pairwise disjoint.
Since $\Phi_n(\xi)$ is the first rational number
satisfying \eqref{eq2.1}, we have
\[
\Phi_n(\xi)=\sum_{k=1}^{\infty}q_k\mathbf 1_{B_{k,n}}(\xi).
\]
Hence $\Phi_n$ is measurable.

We claim that
$
\Phi_n\in L^p(X,d\mu).
$
By the convexity of $t\mapsto |t|^p$ with $p>1$, it follows that
$$
|\Phi_n(\xi)v(r)|^p
\le
2^{p-1}|u(\xi,r)-\Phi_n(\xi)v(r)|^p
+
2^{p-1}|u(\xi,r)|^p.
$$
This together with \eqref{eq2.1} implies that
\begin{align*}
&\bigl|\Phi_n(\xi)\bigr|^p
\int_0^\infty v(r)^p w(r)dr\\
&\le
2^{p-1}
\int_0^\infty
\bigl|u(\xi,r)-\Phi_n(\xi)v(r)\bigr|^p w(r)dr+
2^{p-1}
\int_0^\infty
\bigl|u(\xi,r)\bigr|^p w(r)dr\\
&\le
2^{p-1}
\inf_{c\in\mathbb R}
\int_0^\infty
\bigl|u(\xi,r)-cv(r)\bigr|^p w(r)dr+
\frac{2^{p-1}}{n}
+
2^{p-1}
\int_0^\infty
\bigl|u(\xi,r)\bigr|^p w(r)dr\\
&\le
2^p
\int_0^\infty
\bigl|u(\xi,r)\bigr|^p w(r)dr
+
\frac{2^{p-1}}{n}.
\end{align*}
Integrating over $X$, we obtain
$\Phi_n\in L^p(X,d\mu)$ by the assumptions.

Integrating \eqref{eq2.1} over $X$ and taking infimum over $\Phi$, we get
\begin{align*}
&\inf_{\Phi\in L^p(X,d\mu)}
\int_X\int_0^\infty
\bigl|u(\xi,r)-\Phi(\xi)v(r)\bigr|^p
w(r)drd\mu(\xi)\\
&\le
\int_X
\inf_{c\in\mathbb R}
\int_0^\infty
\bigl|u(\xi,r)-cv(r)\bigr|^p
w(r)drd\mu(\xi)
+
\frac{\mu(X)}{n}.
\end{align*}
Letting $n\to\infty$, one obtains one direction
of the desired equality.
Conversely, for $\Phi\in L^p(X,d\mu)$,
\begin{align*}
\inf_{c\in\mathbb R}
\int_0^\infty
\bigl|u(\xi,r)-cv(r)\bigr|^p w(r)dr
\le
\int_0^\infty
\bigl|u(\xi,r)-\Phi(\xi)v(r)\bigr|^p w(r)dr.
\end{align*}
Integrating over $X$ and taking the infimum over $\Phi$ gives the
reverse inequality.
\end{proof}

\section{Weighted Poincar\'e inequalities for Gamma measures}\label{s2-1}

We establish a weighted $L^p$-Poincar\'e inequality for the
Gamma measure $\mu_\nu$ for every $\nu>0$.
See also \cite{LLRG,LLRBeckner} for weighted Poincar\'e inequalities
and their stability estimates for Gaussian-type measures.
We begin with Muckenhoupt's weighted Hardy criterion
\cite[Theorem~1]{Muckenhoupt}.

\begin{lemma}\label{lem}
Let $1<p<\infty$, $0<L\le\infty$, and $U,V$ be positive
measurable functions on $(0,L)$. Then
\begin{align*}
\bigg(
\int_0^L
\bigg|
U(s)\int_0^s f(t)dt
\bigg|^pds
\bigg)^{\frac1p}
\le
C
\bigg(
\int_0^L
\bigl|V(s)f(s)\bigr|^pds
\bigg)^{\frac1p}
\end{align*}
holds for all nonnegative measurable functions $f$ if and only if
\begin{align*}
B_L
=
\sup_{0<\rho<L}
\bigg(
\int_\rho^L U(s)^pds
\bigg)^{\frac1p}
\bigg(
\int_0^\rho V(s)^{-p'}ds
\bigg)^{\frac1{p'}}
<
\infty,
\end{align*}
where
$
p'
=
\frac{p}{p-1}.
$
Moreover, if $C$ is the least constant for which the above inequality
holds, then
\begin{align*}
B_L
\le C
\le
p^{\frac1p}(p')^{\frac1{p'}}B_L.
\end{align*}
When $L=\infty$, the integrals with upper endpoint $L$ are understood
as integrals over $(0,\infty)$.
\end{lemma}

The following two variants of Lemma \ref{lem} will be used on
$(s_0,\infty)$ and $(0,s_0)$. The first follows by translation and the second by reversing the interval.

\begin{lemma}\label{l1}
Let $1<p<\infty$, $s_0>0$, and $U,V$ be positive measurable
functions on $(s_0,\infty)$. If
$$
B_+
=
\sup_{\rho>s_0}
\bigg(
\int_\rho^\infty U(s)^p ds
\bigg)^{\frac1p}
\bigg(
\int_{s_0}^{\rho} V(s)^{-p'} ds
\bigg)^{\frac1{p'}}
<\infty,
$$
where $p'=\frac{p}{p-1}$, then there exists $C_+>0$ such that
$$
\int_{s_0}^\infty
\bigg|
U(s)\int_{s_0}^s f(t)dt
\bigg|^p ds
\le
C_+
\int_{s_0}^\infty |V(s)f(s)|^p ds
$$
for all nonnegative measurable functions $f$.

Similarly, if $U,V$ are positive measurable functions on $(0,s_0)$ and
$$
B_-
=
\sup_{0<\rho<s_0}
\bigg(
\int_0^\rho U(s)^p ds
\bigg)^{\frac1p}
\bigg(
\int_\rho^{s_0} V(s)^{-p'} ds
\bigg)^{\frac1{p'}}
<\infty,
$$
then there exists $C_->0$ such that
$$
\int_0^{s_0}
\bigg|
U(s)\int_s^{s_0} f(t)dt
\bigg|^p ds
\le
C_-
\int_0^{s_0} |V(s)f(s)|^p ds
$$
for all nonnegative measurable functions $f$.

Moreover, the constants may be chosen so that
$$
C_+
\le
p(p')^{p-1}B_+^p
\quad\text{and}\quad
C_-
\le
p(p')^{p-1}B_-^p.
$$
\end{lemma}

\begin{proof}
For the first assertion, set
$
\theta=s-s_0
$
for $s\in(s_0,\infty)$, and let
$
\widetilde f(\cdot)=f(s_0+\cdot),\
\widetilde U(\cdot)=U(s_0+\cdot),\
\widetilde V(\cdot)=V(s_0+\cdot).
$
Then for $\theta>0$
$$
\int_{s_0}^{s_0+\theta}f(t)dt
=
\int_0^\theta\widetilde f(\tau)d\tau.
$$
Moreover,
$$
\int_{s_0}^\infty
\bigg|
U(s)\int_{s_0}^s f(t)dt
\bigg|^pds=
\int_0^\infty
\bigg|
\widetilde U(\theta)
\int_{s_0}^{s_0+\theta} f(t)dt
\bigg|^pd\theta
=
\int_0^\infty
\bigg|
\widetilde U(\theta)
\int_0^\theta \widetilde f(\tau)d\tau
\bigg|^pd\theta,
$$
$$
\int_{s_0}^\infty |V(s)f(s)|^pds
=
\int_0^\infty |\widetilde V(\theta)\widetilde f(\theta)|^pd\theta.
$$
It follows that
\begin{align*}
\sup_{\rho>0}\!
\bigg(\!
\int_\rho^\infty \widetilde U(\xi)^p d\xi
\!\bigg)^{\frac1p}
\bigg(\!
\int_0^\rho \widetilde V(\xi)^{-p'} d\xi
\!\bigg)^{\frac1{p'}}
&\!\!=\!
\sup_{\rho>0}
\bigg(
\int_\rho^\infty U(s_0+\xi)^p d\xi
\bigg)^{\frac1p}
\bigg(
\int_0^\rho V(s_0+\xi)^{-p'} d\xi
\bigg)^{\frac1{p'}}
\\
&\!\!=\!
\sup_{\rho>s_0}
\bigg(
\int_\rho^\infty U(s)^p ds
\bigg)^{\frac1p}
\bigg(
\int_{s_0}^\rho V(s)^{-p'} ds
\bigg)^{\frac1{p'}}
\\
&\!\!=\!B_+.
\end{align*}
Hence, the first assertion follows from Lemma \ref{lem}, with
$
C_+
\le
p(p')^{p-1}B_+^p.
$

For the second assertion, set
$
\theta=s_0-s
$
for $s\in(0,s_0)$, and let
$
\widetilde f(\cdot)=f(s_0-\cdot),\
\widetilde U(\cdot)=U(s_0-\cdot),\
\widetilde V(\cdot)=V(s_0-\cdot).
$
Then for $0<\theta<s_0$,
$$
\int_s^{s_0}f(t)dt
=
\int_0^\theta\widetilde f(\tau)d\tau.
$$
Moreover,
$$
\int_0^{s_0}
\bigg|
U(s)\int_s^{s_0} f(t)dt
\bigg|^pds=
\int_0^{s_0}
\bigg|
\widetilde U(\theta)
\int_{s_0-\theta}^{s_0}f(t)dt
\bigg|^pd\theta
=
\int_0^{s_0}
\bigg|
\widetilde U(\theta)
\int_0^\theta \widetilde f(\tau)d\tau
\bigg|^pd\theta,
$$
$$
\int_0^{s_0}|V(s)f(s)|^pds
=
\int_0^{s_0}|\widetilde V(\theta)\widetilde f(\theta)|^pd\theta.
$$
It follows that
\begin{align*}
\sup_{0<\rho<s_0}\!\!
\bigg(\!
\int_\rho^{s_0}\widetilde U(\xi)^p d\xi
\!\bigg)^{\frac1p}
\bigg(\!
\int_0^\rho \widetilde V(\xi)^{-p'} d\xi
\!\bigg)^{\frac1{p'}}
&\!\!=\!
\sup_{0<\rho<s_0}
\bigg(
\int_\rho^{s_0}U(s_0\!-\!\xi)^p d\xi
\bigg)^{\frac1p}
\bigg(
\int_0^\rho V(s_0\!-\!\xi)^{-p'} d\xi
\bigg)^{\frac1{p'}}
\\
&\!\!=\!
\sup_{0<\rho<s_0}
\bigg(
\int_0^\rho U(s)^p ds
\bigg)^{\frac1p}
\bigg(
\int_\rho^{s_0}V(s)^{-p'} ds
\bigg)^{\frac1{p'}}
\\
&\!\!=\!B_-.
\end{align*}
Hence, the second assertion follows from Lemma \ref{lem} with
$L=s_0,$
and
$C_-\leq p(p')^{p-1}B_-^p.$
\end{proof}

Let $W_{\mathrm{loc}}^{1,2}(0,\infty)$ be the space of functions whose values and weak first derivatives are square integrable on every compact subinterval of $(0,\infty)$.
Let $\mathcal D_\nu=\big\{
h\in W_{\mathrm{loc}}^{1,2}(0,\infty):\int_0^\infty
\big(|h(s)|^2+s|h'(s)|^2\big)
s^{\nu-1}e^{-s}ds<\infty\big\}.$

We shall need the following weighted $L^p$-Poincar\'e inequality for
the Gamma measure. See \cite{BJM} for related weighted Poincar\'e
inequalities and spectral gap estimates for one-dimensional diffusions.

\begin{lemma}\label{l2}
Let $1<p<\infty$ and $\nu>0$. Then there exists a constant
$P_{\nu,p}>0$ such that, for $h\in C_c^\infty(0,\infty)$,
\begin{align}\label{f1}
\inf_{c\in\mathbb R}
\int_0^\infty
|h(s)-c|^p s^{\nu-1}e^{-s}ds
\le
P_{\nu,p}
\int_0^\infty
s|h'(s)|^p s^{\nu-1}e^{-s}ds.
\end{align}
Moreover, when $p=2$, the sharp constant is $ P_{\nu,2}=1.$
The inequality \eqref{f1} for $p=2$
can be extended to every
$h\in\mathcal D_\nu,$
with equality if and only if
\begin{align*}
h(s)=c_0+c_1\bigl(s-\nu\bigr)
\end{align*}
for some $c_0,c_1\in\mathbb R$.
\end{lemma}

\begin{proof}
We first show \eqref{f1} by using Lemma \ref{l1}.
Fix $s_0=1$, and let
\[
U(s)=s^{\frac{\nu-1}{p}}e^{-\frac{s}{p}},\quad V(s)=s^{\frac{\nu}{p}}e^{-\frac{s}{p}}.
\]
Then the corresponding Hardy quantities are
\begin{align*}
B_+&=\sup_{\rho>1}
\bigg(\int_\rho^\infty s^{\nu-1}e^{-s}ds\bigg)^{\frac{1}{p}}
\bigg(\int_1^\rho s^{-\frac{\nu}{p-1}}e^{\frac{s}{p-1}}ds\bigg)^{\frac{1}{p'}},\\
B_-&=\sup_{0<\rho<1}
\bigg(\int_0^\rho s^{\nu-1}e^{-s}ds\bigg)^{\frac{1}{p}}
\bigg(\int_\rho^1 s^{-\frac{\nu}{p-1}}e^{\frac{s}{p-1}}ds\bigg)^{\frac{1}{p'}}.
\end{align*}

For $\rho$ big enough, there exists $C$ (may depend on $\nu$ and $p$, but does not depend on $\rho$) such that
\begin{equation}\label{eq2.3}
\int_\rho^\infty s^{\nu-1}e^{-s}ds\le C\rho^{\nu-1}e^{-\rho},\quad
\int_1^\rho s^{-\frac{\nu}{p-1}}e^{\frac{s}{p-1}}ds
\le C\rho^{-\frac{\nu}{p-1}}e^{\frac{\rho}{p-1}}.
\end{equation}
The first estimate is the standard Gamma tail bound (see, e.g., \cite[Chapter 6]{AS}).
The second estimate follows by differentiating $s^{-\frac{\nu}{p-1}}e^{\frac{s}{p-1}}$ for sufficiently large $s$.
In fact, letting $g(s)=s^{-\frac{\nu}{p-1}}e^{\frac{s}{p-1}},s>0$, its derivative satisfies
$$\frac{g'}{g}=\frac{1}{p-1}\Big(1-\frac{\nu}{s}\Big)\geq\frac{1}{2(p-1)},$$
if $s>S$ for some $S>1$.
Since $c_0=\int_1^{S} g(s)ds<\infty$ and $\int_{S} ^\rho g(s)ds
\leq 2(p-1)\int_{S} ^\rho g'(s)ds$,
it holds
\begin{align*}
\int_1^\rho g(s)ds\leq c_0+2(p-1)(g(\rho)-g(S))\leq c_0+ 2(p-1)g(\rho)\leq Cg(\rho),
\end{align*}
for sufficiently large $\rho$.
As a consequence of \eqref{eq2.3}, we get $\big(\int_\rho^\infty s^{\nu-1}e^{-s}ds\big)^{\frac{1}{p}}
\big(\int_1^\rho s^{-\frac{\nu}{p-1}}e^{\frac{s}{p-1}}ds\big)^{\frac{1}{p'}}\leq C \rho^{-\frac{1}{p}}$ as $\rho\to\infty$. Clearly, this product converges to zero as $\rho\to 1^+$. Thus, $B_+ <\infty$.

For $0<\rho<1$, we have $\int_0^\rho s^{\nu-1}e^{-s}ds\leq  \int_0^\rho s^{\nu-1}ds=\frac{\rho ^{\nu}}{\nu}$,
and $\int_\rho^1 s^{-\frac{\nu}{p-1}}e^{\frac{s}{p-1}}ds\leq e^{\frac{1}{p-1}}\int_\rho^1 s^{-\frac{\nu}{p-1}}ds$.
Hence, the product defining $B_-$ is bounded by
\[
C\begin{cases}
\rho^{\frac{\nu}{p}},&\nu<p-1,\\
\rho^{\frac{\nu}{p}}|\log(\rho)|^{\frac{1}{p'}},&\nu=p-1,\\
\rho^{\frac{p-1}{p}},&\nu>p-1.
\end{cases}
\]
These bounds tend to zero as $\rho\rightarrow0^+$, and remain bounded as $\rho\rightarrow1^-$.
The continuity shows $B_-<\infty$.
Since $|h(s)-h(1)|\le\bigl|\int_1^s|h'(t)|dt\bigr|$, it yields \eqref{f1} by applying Lemma~\ref{l1} to $f=|h'|$, with
$P_{\nu,p}=p(p')^{p-1}\max\{B_+^p,B_-^p\}$.

Let us consider the case $p=2$. Let $d\mu_\nu(s)=\frac{s^{\nu-1}e^{-s}}{\Gamma(\nu)}ds=w_\nu (s)ds$ be the normalized Gamma measure, and let $\overline h_\nu=\int_0^\infty h\,d\mu_\nu$.
For $h\in L^2(\mu_\nu)$,
\[
\int_0^\infty |h-c|^2\,d\mu_\nu
=\int_0^\infty|h-\overline h_\nu|^2d\mu_\nu+|\overline h_\nu-c|^2.
\]
Then $\inf_{c\in\mathbb R}\int |h-c|^2\,d\mu_\nu
=\operatorname{Var}_{\mu_\nu}(h).$
Therefore, the required inequality is equivalent to
\begin{equation}\label{eq2.4}
\operatorname{Var}_{\mu_\nu}(h)
\le
\int_0^\infty s|h'(s)|^2\,d\mu_\nu(s).
\end{equation}

Let us show \eqref{eq2.4} by using the generalized Laguerre polynomials.
For each integer $n\geq0$, the generalized Laguerre polynomial is defined by
\[
P_n=L_n^{(\nu-1)}(s)=\frac{s^{1-\nu}e^s}{n!}
\frac{d^n}{ds^n}\bigl(e^{-s}s^{n+\nu-1}\bigr).
\]
These polynomials form a complete orthogonal basis of $L^2(\mu_\nu)$ and satisfy
$sP_n''+(\nu-s)P_n'=-nP_n.$
See, e.g., \cite[Section~5.1 and Theorem~5.7.1]{Szego} for the Laguerre polynomial identities and completeness.
Therefore, if we define the operator $L_\nu h=sh''+(\nu-s)h'$, one has
$L_\nu P_n=-nP_n$.
By the definition, we have
$P_0=1,~P_1=\nu-s.$
For $h\in\mathcal D_\nu$, we have
$$h=\sum_{n=0}^\infty a_nP_n=\overline h_\nu+\sum_{n=1}^\infty a_nP_n, \quad a_n=\frac{\int_0^\infty hP_n\,d\mu_\nu}
{\|P_n\|_{L^2(\mu_\nu)}^2}.$$
By the orthogonality, we have
\begin{equation}\label{eq2.5}
\operatorname{Var}_{\mu_\nu}(h)
=\int_0^\infty|h-\overline h_\nu|^2d\mu_\nu
=\sum_{n\ge1}a_n^2\|P_n\|_{L^2(\mu_\nu)}^2.\end{equation}

Since $L_\nu g=sg''+(\nu-s)g'=
\frac1{w_\nu}(sw_\nu g')'$ where $w_\nu$ is the density of $\mu_\nu$,
the integration by parts for the polynomials gives, for $m,n\ge1$,
\[
\int_0^\infty sP_m'P_n'\,d\mu_\nu
=-\int_0^\infty P_mL_\nu P_n\,d\mu_\nu
=n\delta_{mn}\|P_n\|_{L^2(\mu_\nu)}^2.
\]
The boundary terms vanish since
$sw_\nu(s)=s^\nu e^{-s}/\Gamma(\nu)$ tends to zero
as $s\to0^+$, while the exponential factor dominates
the polynomial factors as $s\to\infty$.
We define
\[
q_n=\frac{P_n'}{\sqrt n\,\|P_n\|_{L^2(\mu_\nu)}},
\quad n\geq1.
\]
Then $\{q_n\}_{n\geq1}$ is an orthonormal basis of
$L^2((0,\infty),s\,d\mu_\nu)$.
Indeed, its completeness follows from
\[P_n'=-L_{n-1}^{(\nu)},\quad s\,d\mu_\nu=\nu\,d\mu_{\nu+1},\]
and the completeness of the generalized Laguerre polynomials
in $L^2(\mu_{\nu+1})$.
For $h\in\mathcal D_\nu$, we set
\[\mathcal E_\nu(h)=\int_0^\infty s|h'|^2\,d\mu_\nu.\]
Since $h$ is locally absolutely continuous, the Cauchy-Schwarz inequality implies that
\[
|h(s)|
\le |h(1)|
+\bigl(\Gamma(\nu)\mathcal E_\nu(h)\bigr)^{\frac12}
\left(
\int_{\min\{s,1\}}^{\max\{s,1\}}
t^{-\nu}e^t\,dt
\right)^{\frac12}<\infty.
\]
Then it follows that
$\lim\limits_{s\to0^+}sw_\nu(s)h(s)P_n'(s)=
\lim\limits_{s\to\infty}sw_\nu(s)h(s)P_n'(s)=0$, for every $n\ge1$.
From integration by parts, we get
\[
\int_0^\infty sh'P_n'\,d\mu_\nu
=-\int_0^\infty hL_\nu P_n\,d\mu_\nu
=n a_n\|P_n\|_{L^2(\mu_\nu)}^2.
\]
Therefore,
\[
\int_0^\infty sh'q_n\,d\mu_\nu
=\sqrt n\,a_n\|P_n\|_{L^2(\mu_\nu)}.
\]
Applying Parseval's identity to $h'$ in
$L^2((0,\infty),s\,d\mu_\nu)$, we obtain
\begin{equation}\label{eq2.6}
\mathcal E_\nu(h)
=\sum_{n\ge1}
\left|\int_0^\infty sh'q_n\,d\mu_\nu\right|^2
=\sum_{n\ge1}n a_n^2
\|P_n\|_{L^2(\mu_\nu)}^2.
\end{equation}
By \eqref{eq2.5} and \eqref{eq2.6}, we get $\operatorname{Var}_{\mu_\nu}(h)\le\mathcal E_\nu(h)$, with equality precisely when $a_n=0$ for $n\ge2$.
As $P_0=1$ and $P_1=\nu-s$, this gives $h=c_0+c_1(s-\nu)$.

To verify sharpness on $C_c^\infty(0,\infty)$ for every $\nu>0$, choose $0\le\eta_j\le1$ with
$\eta_j=1$ on $[j^{-1},j]$, supported in $[(2j)^{-1},2j]$, and
$|\eta_j'|\le Cj$ near zero and $|\eta_j'|\leq C j^{-1}$ near infinity.
For $h_j(s)=s\eta_j(s)$, we have that $
h_j'=\eta_j+s\eta_j'$. Thus $h_j'=1$ on $[1/j,j]$, and $
|h_j'-1|
\le C\,\mathbf\chi_{(0,1/j)\cup(j,\infty)}$ on $(0,1/j)\cup(j,\infty)$.
Then
\[
h_j\rightarrow s\ \hbox{in }L^2(\mu_\nu)\text{ as }j\to\infty,\quad
\int_0^\infty s\bigl|h_j'(s)-1\bigr|^2d\mu_\nu(s)\rightarrow0\ \text{as }j\to\infty.
\]
The contribution from $(0,j^{-1})$ is bounded by
$C\int_0^{1/j}s^\nu ds$, while the contribution from
$(j,\infty)$ converges to zero as $j\to\infty$.
Since $\operatorname{Var}_{\mu_\nu}(s)=\int s^2d\mu_\nu(s)-(\int sd\mu_\nu(s))^2=\nu$ and $\mathcal E_\nu(s)=\nu$, we have
\[
\frac{\operatorname{Var}_{\mu_\nu}(h_j)}
{\mathcal E_\nu(h_j)}
\rightarrow 1
\quad\text{as }j\to\infty.
\]
Hence $P_{\nu,2}=1$ is sharp.
\end{proof}

\section{Sharp \texorpdfstring{$L^2$}{L2} stability of anisotropic radial inequality}\label{s-sharp}

This section is devoted to
the sharp $L^2$ stability of anisotropic radial inequality.
The proof consists of three steps: reducing the parameter region,
showing that the functions attaining equality belong to the completion space,
and minimizing over the common scale.

\subsection{Sharp stability of $L^2$ CKN inequality}

We consider the following two parameter regions:
\[
\mathcal A\!=\!\{(a,b)\!\in\!\mathbb R^2\!:\!b+1-a\ne0,(b+1-a)(N-2b-2)\!\ge\!0\},
\mathcal B\!=\!\{(a,b)\!\in\!\mathbb R^2\!:\!(b+1-a)(N-2b-2)\!<\!0\}.
\]
It follows from \eqref{eq:unified-parameters} that
$\gamma=0$ on $\mathcal A$ and $\gamma=2b+2-N$ on $\mathcal B$.

The following lemma gives a reduction from the case of $\mathcal B$ to the case $\mathcal A$.

\begin{lemma}\label{sharp:conjugation}
Let $K\in\mathcal K^{ss}_{(o)}$.
For $(a,b)\in\mathcal B$, we write \[a_*=a-(2b+2-N),\ b_*=b-(2b+2-N).\]
Then $(a_*,b_*)\in\mathcal A$, and the parameters $\nu$ and $b+1-a$
are unchanged under the transforms $(a,b)\mapsto (a_*,b_*)$. The map
$T:\mathcal C^2_{K,a_*,b_*}\rightarrow\mathcal C^2_{K,a,b},$
defined by $T v=r^{2b+2-N}v$, is an isometric isomorphism such that
for $v\in \mathcal C^2_{K,a_*,b_*}$,
\begin{equation}\label{sharp:conjugation-norms}
I_a(Tv)=I_{a_*}(v),\quad D_b(Tv)=D_{b_*}(v),\quad
I_{a,b}(Tv)=I_{a_*,b_*}(v).
\end{equation}
And $T$ maps the
extremal family for $(a_*,b_*)$ onto $\mathcal M_{K,a,b}$.
\end{lemma}

\begin{proof}
By the definition of $a_*,b_*$, we have
$b_*+1-a_*=b+1-a$ and $N-2b_*-2=-\bigl(N-2b-2\bigr)$, and then
$(a_*,b_*)\in\mathcal A$.
The transformation also preserves $\nu=1+\frac{|N-2b-2|}{|b+1-a|}$.
We denote by $k=2b+2-N$ for brevity.
For a smooth $v\in C^\infty_c(\mathbb R^N\setminus\{o\})$,
it follows from the anisotropic polar formula that
\[
D_b(r^kv)=\int_{\partial K}\int_0^\infty
\bigg|\partial_rv+\frac{k}{r}v\bigg|^2r^{k+1}\,dr\,d\mu_K.
\]
For each fixed angle, it follows from integration by parts that
\[
k\int_0^\infty 2v(\partial_r v)r^k\,dr=
k\int_0^\infty\partial_r(v^2)r^k\,dr
=-k^2\int_0^\infty v^2r^{k-1}\,dr.
\]
Therefore,
\[
D_b(r^kv)=\int_{\partial K}\int_0^\infty
\bigg|\partial_rv+\frac{k}{r}v\bigg|^2r^{k+1}\,dr\,d\mu_K=\int_{\partial K}\int_0^\infty
(\partial_rv)^2r^{k+1}\,dr\,d\mu_K=D_{b_*}(v).
\]
Similarly, we get the other identities in \eqref{sharp:conjugation-norms},
by substituting $u=r^kv$ into $I_a(u)$ and $I_{a,b}(u)$.

Since $r=\|x\|_K$ is $C^1$ away from the origin,
$r^kv$ has compact support in an annulus and belongs to $W^{1,2}$ there.
Choose smooth functions $f_\ell$, compactly supported in a slightly larger annulus, such that $f_\ell\to r^kv$ in $W^{1,2}$. The estimate
\[
\bigl|\mathcal R_K(f)\bigr|\le C_K|\nabla f|,
\quad C_K=\max_{\sigma\in\partial K}|\sigma|<\infty,
\]
implies $\|f_\ell-r^kv\|_{\mathcal C^2_{K,a,b}}\leq C \|f_\ell-r^kv\|_{W^{1,2}}\to 0$.
Hence $r^kv$ be in $\mathcal C^2_{K,a,b}$.
The first two identities show $\|Tv\|_{\mathcal C^2_{K,a,b}}
=\|v\|_{\mathcal C^2_{K,a_*,b_*}}.$
The same argument for multiplication by $r^{-k}$ gives its inverse.
Therefore, $T$ can be extended isometrically to all $v\in \mathcal C^2_{K,a_*,b_*}$. In this extension, we also have $Tv=r^k v$ almost everywhere.
Indeed, choose a sequence $v_\ell$ of test functions converging to $v$ in
$\mathcal C^2_{K,a_*,b_*}$ as $\ell\to\infty$. Then both
$v_\ell\rightarrow v$ and $Tv_\ell\rightarrow Tv$ locally in $L^2$
as $\ell\to\infty$. Thus $Tv=r^kv$ almost everywhere.
Then one has the weighted $L^2$ identity.
The statement about the extremal families follows
from $g_\eta(r)=r^ke^{-\frac{\eta r^{b+1-a}}{|b+1-a|}}$ in $\mathcal B$.
\end{proof}

For $(a,b)\in\mathcal A$, we have $N-a-b-1=\nu\bigl(b+1-a\bigr)$.
It follows from integration by parts
and the Cauchy-Schwarz inequality that
\begin{align}\nonumber
\nu\bigl|b+1-a\bigr|I_{a,b}(u)
&=\bigl|N-a-b-1|\int_{\partial K}\int_0^\infty u^2 r^{N-a-b-2}drd\mu_K(\sigma)\\\nonumber
&=-2{\rm sgn}(N-a-b-1)\int_{\partial K}\int_0^\infty u \partial_r(u) r^{N-a-b-1}drd\mu_K(\sigma)\\
&\leq2\sqrt{D_b(u)I_a(u)}\leq D_b(u)+I_a(u),\label{sharp:embedding}
\end{align}
for every
$u\in C_c^\infty(\mathbb R^N\setminus\{o\})\setminus\{0\}$.
This proves the continuous embedding of $\mathcal C^2_{K,a,b}$
into $L^2(\RN,r^{-a-b-1}dx)$, if $(a,b)\in\mathcal A$.
Lemma~\ref{sharp:conjugation} transfers this statement to
$\mathcal B$. In particular, all deficits and distances in
Theorem~\ref{thm:unified} are well defined on $\mathcal C^2_{K,a,b}$.

We recall that for $\lambda>0$,
$g_\lambda(r)=r^\gamma e^{-\frac{\lambda r^{b+1-a}}{|b+1-a|}}$, and $d\mu_\nu(s)=\frac{s^{\nu-1}e^{-s}}{\Gamma(\nu)}\,ds$, for $s>0$.

\begin{lemma}\label{sharp:density}
For every $\lambda>0$ and $\Phi_0,\Phi_1\in L^2(\partial K,d\mu_K)$,
the function
\[
u(r,\sigma)=g_\lambda(r)h\left(\frac{2\lambda}{|b+1-a|} r^{b+1-a},\sigma\right),\quad
h(s,\sigma)=\Phi_0(\sigma)+\Phi_1(\sigma)\bigl(s-\nu\bigr),
\]
belongs to $\mathcal C^2_{K,a,b}$.
\end{lemma}

\begin{proof}
By Lemma~\ref{sharp:conjugation}, it suffices to prove the result
for $(a,b)\in\mathcal A$,
where $\gamma=0$ and $\nu\ge1$.
By the change of variables $s=\frac{2\lambda}{|b+1-a|} r^{b+1-a}$, we have
\begin{align}\label{eq4.4}
I_a(u)&=\frac{\Gamma(\nu)}{2\lambda }\left(\frac{2\lambda}{|b+1-a|}\right)^{-\nu}
\int_{\partial K}\int_0^\infty s|h|^2\,d\mu_\nu\,d\mu_K,\\
D_b(u)&=\frac{\lambda \Gamma(\nu)}{2}\left(\frac{2\lambda}{|b+1-a|}\right)^{-\nu}
\int_{\partial K}\int_0^\infty s|2\partial_sh-h|^2\,d\mu_\nu\,d\mu_K.\label{eq4.5}
\end{align}
By the finiteness of polynomial moments of the Gamma measure, we get $I_a(u)<+\infty$ and $D_b(u)<+\infty$.

By the Stone-Weierstrass theorem, the restrictions of polynomials on
$\mathbb R^N$ are uniformly dense in $C(\partial K)$.
Since $C(\partial K)$ is dense in $L^2(\partial K,d\mu_K)$ and $\mu_K(\partial K)<\infty$, uniform convergence implies
convergence in $L^2(\partial K,d\mu_K)$.
Therefore, for $i=0,1$, we may choose
$G_{i,j}\in C^\infty(\mathbb R^N)$ such that
\[
\Phi_{i,j}:=G_{i,j}|_{\partial K}
\longrightarrow \Phi_i
\quad\text{in }L^2(\partial K,d\mu_K)
\quad\text{as }j\to\infty.
\]
Let $h_j(s,\sigma)=\Phi_{0,j}(\sigma)+\Phi_{1,j}(\sigma)\bigl(s-\nu\bigr),$
$u_j(r,\sigma)=g_\lambda(r)h_j\left(\frac{2\lambda}{|b+1-a|} r^{b+1-a},\sigma\right)$.
Applying \eqref{eq4.4} and \eqref{eq4.5} to $u_j-u$, and using
the finiteness of the polynomial moments of the Gamma measure,
we obtain
\[
I_a(u_j-u)+D_b(u_j-u)
\le C\sum_{i=0}^1\|\Phi_{i,j}-\Phi_i\|_2^2\to 0.
\]
Thus $u_j\to u$ in the weighted energy norm.

For $0<\varepsilon<1$, choose a smooth cutoff $\chi_\varepsilon$
equal to one on $[\varepsilon,\varepsilon^{-1}]$, equal to zero
outside $(\varepsilon^2,\varepsilon^{-2})$, with values in $[0,1]$ and
\[
\bigl|\chi_\varepsilon'(s)\bigr|\le\frac{C}{s|\log(\varepsilon)|},
\]
for some constant $C>0$ independent of $\varepsilon$.
Replace $h_j$ by $\chi_\varepsilon h_j$.
In the derivative energy, the additional term is
$2\chi_\varepsilon'h_j$.
All other differences converge to zero as $\varepsilon\to0$
by the dominated convergence theorem.
For fixed $j$, there exists $C_j>0$ such that
the contribution of $2\chi_\varepsilon'h_j$ to the derivative
energy near zero is bounded by
\[
\frac{C_j}{|\log(\varepsilon)|^2}
\int_{\varepsilon^2}^{\varepsilon}s^{\nu-2}\,ds\rightarrow0\quad\text{as }\varepsilon\to0.
\]
At the endpoint $\nu=1+\frac{N-2b-2}{b+1-a}=1$, this bound is $\frac{C_j}{|\log(\varepsilon)|}$.
At infinity, the corresponding bound is
\[
\frac{C_j}{|\log(\varepsilon)|^2}
\int_{\varepsilon^{-1}}^\infty\bigl(1+s^2\bigr)s^{\nu-2}e^{-s}\,ds
\rightarrow0\quad\text{as }\varepsilon\to0.
\]

For fixed $j$ and $\varepsilon$, the cut-off function is $C^1$
and compactly supported in an annulus. Choose smooth functions
supported in a slightly larger annulus that converge to this
function in $W^{1,2}$. Since the weights are bounded on that
annulus and $|\mathcal R_K(f)|\le C_K|\nabla f|$, this convergence
also holds in the weighted energy norm.
By a diagonal argument, there exists a subsequence
$w_n\in C_c^\infty(\RN\setminus\{o\})$ such that
$w_n\rightarrow u$ in the weighted energy norm as $n\to\infty$.
\end{proof}

The next proposition gives the minimizing scales of  functions to the fixed-scale extremal subspace,
when the two angular
coefficients are orthogonal.

\begin{proposition}\label{prop:nearest-scales}
Let $\lambda>0$ and
$u=g_\lambda\left[\Phi_0+\Phi_1\left(\frac{2\lambda}{|b+1-a|} r^{b+1-a}-\nu\right)\right],$
where $\Phi_0,\Phi_1\in L^2(\partial K,d\mu_K)$ are not both zero and
$\langle\Phi_0,\Phi_1\rangle=0$.
If $\|\Phi_0\|_2^2\ge\nu \|\Phi_1\|_2^2$, the unique nearest scale is $\eta=\lambda$.
If $\|\Phi_0\|_2^2<\nu \|\Phi_1\|_2^2$, there are two nearest scales $\eta_\pm=\lambda\frac{1\mp z_*}{1\pm z_*}$, where
$z_*=\sqrt{\frac{\nu \|\Phi_1\|_2^2-\|\Phi_0\|_2^2}{\nu\bigl(\nu+1\bigr)\|\Phi_1\|_2^2}}$.
More explicitly,
\begin{equation}\label{sharp:two-mode-distance}
d_{K,a,b}(u,\mathcal M_{K,a,b})^2=\frac{\Gamma(\nu)}{|b+1-a|}\left(\frac{2\lambda}{|b+1-a|}\right)^{-\nu}\bigl(\|\Phi_0\|_2^2+\nu \|\Phi_1\|_2^2-Q_{\max}\bigr),
\end{equation}
where
\[Q_{\max}=
\begin{cases}
\|\Phi_0\|_2^2,&\|\Phi_0\|_2^2\ge\nu \|\Phi_1\|_2^2,\\[2pt]
\displaystyle\frac{\bigl(\|\Phi_0\|_2^2+\nu^2\|\Phi_1\|_2^2\bigr)^{\nu+1}}
{\nu^\nu\bigl(\nu+1\bigr)^{\nu+1}\|\Phi_1\|_2^{2\nu}},&\|\Phi_0\|_2^2<\nu \|\Phi_1\|_2^2.
\end{cases}
\]
\end{proposition}

Clearly, if $\|\Phi_0\|_2^2<\nu \|\Phi_1\|_2^2$, then the product of two nearest scales is $\eta_+\eta_-=\lambda^2.$

\begin{proof}
For fixed $\eta>0$, the functions $\Phi(\sigma)g_\eta(r)$ form a
closed subspace of the weighted $L^2$ space.
The minimizing angular coefficient is given by orthogonal projection:
\begin{equation}\label{sharp:projection-definition}
\Phi_\eta(\sigma)=
\frac{\int_0^\infty
u(r,\sigma)g_\eta(r)r^{N-a-b-2}dr}
{\int_0^\infty
g_\eta(r)^2r^{N-a-b-2}dr}.
\end{equation}
Let
$
z=\frac{\lambda-\eta}{\lambda+\eta}\in(-1,1).
$
By the change of variable $t=\frac{(\lambda+\eta) r^{b+1-a}}{|b+1-a|}$,
it follows
\[
\int_0^\infty g_\lambda(r)g_\eta(r)r^{N-a-b-2}dr
=\frac{\Gamma(\nu)}{|b+1-a|}
\biggl(\frac{\lambda+\eta}{|b+1-a|}\biggr)^{-\nu}.
\]
Letting $s=\frac{2\lambda r^{b+1-a}}{|b+1-a|}=\frac{2\lambda }{\lambda+\eta}t$, we obtain from $\Gamma(\nu+1)=\nu\Gamma(\nu)$ that
\begin{align*}
\int_0^\infty s\,g_\lambda(r)g_\eta(r)r^{N-a-b-2}dr
=\frac{2\lambda\nu}{\lambda+\eta} \int_0^\infty g_\lambda(r)g_\eta(r)r^{N-a-b-2}dr.
\end{align*}
And we get
\[
\int_0^\infty
\bigl(s-\nu\bigr)g_\lambda(r)g_\eta(r)r^{N-a-b-2}dr
=
\nu z
\int_0^\infty
g_\lambda(r)g_\eta(r)r^{N-a-b-2}dr,
\]
where $z=\frac{\lambda-\eta}{\lambda+\eta}$.
Therefore, \eqref{sharp:projection-definition} yields
\begin{equation}\label{sharp:projection-coefficient}
\Phi_\eta=\biggl(\frac{2\eta}{\lambda+\eta}\biggr)^\nu
\bigl(\Phi_0+\nu z\Phi_1\bigr).
\end{equation}
Using the mean and variance of the Gamma measure, we have
\begin{align*}
I_{a,b}(u)&=\int_{\partial K}\int_0^\infty g_\lambda^2(\Phi_0^2+(\Phi_1(s-\nu ))^2)r^{N-a-b-2}drd\mu_K\\
&=\frac{\Gamma(\nu)}{|b+1-a|}\left(\frac{2\lambda}{|b+1-a|}\right)^{-\nu}(\|\Phi_0\|_2^2+\nu \|\Phi_1\|_2^2).
\end{align*}
Note that
\begin{align*}
&I_{a,b}\bigl(u-\Phi g_\eta\bigr)\\
&=\!I_{a,b}(u)
-2\int_{\partial K}\!\!\Phi\!\!\int_0^\infty\!\! g_\lambda g_\eta (\Phi_0+\Phi_1(s-\nu ))r^{N-a-b-2}drd\mu_K
+\int_{\partial K}\!\!\Phi^2\!\!\int_0^\infty\!\! g_\eta^2 r^{N-a-b-2}drd\mu_K \\
&=\!I_{a,b}(u)
-2\langle\Phi,\Phi_0+\nu z\Phi_1\rangle\frac{\Gamma(\nu)}{|b+1-a|}\left(\frac{\lambda+\eta}{|b+1-a|}\right)^{-\nu}\!
+\frac{\Gamma(\nu)}{|b+1-a|}\left(\frac{2\eta}{|b+1-a|}\right)^{-\nu}\!\|\Phi\|_2^2.
\end{align*}
Therefore,
\begin{align*}
&\inf_{\Phi\in L^2(\partial K,d\mu_K)}
I_{a,b}\bigl(u-\Phi g_\eta\bigr)\\
&=\!
\frac{\Gamma(\nu)}{|b\!+\!1\!-\!a|}\bigg(\frac{2\lambda}{|b\!+\!1\!-\!a|}\bigg)^{-\nu}
\!\Bigl(
\|\Phi_0\|_2^2+\nu \|\Phi_1\|_2^2\!-\!\bigl(1\!-\!z^2\bigr)^\nu
\bigl(\|\Phi_0\|_2^2\!+\!2\nu z\langle\Phi_0,\Phi_1\rangle\!+\!\nu^2z^2\|\Phi_1\|_2^2\bigr)
\Bigr).
\end{align*}

Under the assumption $\langle\Phi_0,\Phi_1\rangle=0$, minimization over $\eta>0$ is equivalent
to maximizing
\[
Q(z)=\bigl(1-z^2\bigr)^\nu\bigl(\|\Phi_0\|_2^2+\nu^2\|\Phi_1\|_2^2z^2\bigr),\quad -1<z<1.
\]
If $\|\Phi_1\|_2^2=0$, the unique maximum is at $z=0$.
Otherwise,
\[
Q'(z)=2\nu z\bigl(1-z^2\bigr)^{\nu-1}
\bigl[\nu \|\Phi_1\|_2^2-\|\Phi_0\|_2^2-\nu\bigl(\nu+1\bigr)\|\Phi_1\|_2^2z^2\bigr].
\]
If $\|\Phi_0\|_2^2\ge\nu \|\Phi_1\|_2^2$, then $Q'(z)<0$ for $0<z<1$.
Since $Q$ is even, its unique maximum is at zero.
This also covers the threshold case $\|\Phi_0\|_2^2=\nu \|\Phi_1\|_2^2$.
If $\|\Phi_0\|_2^2<\nu \|\Phi_1\|_2^2$, then $Q'(z)>0$ for $0<z<z_*$ and
$Q'(z)<0$ for $z_*<z<1$. Thus, by evenness, the two maxima
are at $z=\pm z_*$, where $0<z_*<1$.
Also $Q(z)\rightarrow0$ as $z\to\pm1$.
Substituting the maximizing values of $z$ gives the stated formulas for the nearest scales and $Q_{\max}$, as desired.
\end{proof}

\begin{remark}\label{rem:quartic-threshold}
When $\|\Phi_0\|_2^2=\nu \|\Phi_1\|_2^2>0$, the nearest scale remains unique, although the coefficient of $z^2$ in the expansion of the normalized squared projection norm is zero:
\begin{equation}\label{sharp:quartic-threshold}
\frac{Q(z)}{\|\Phi_0\|_2^2}=\bigl(1-z^2\bigr)^\nu\bigl(1+\nu z^2\bigr)
=1-\frac{\nu\bigl(\nu+1\bigr)}2z^4+O(z^6),\quad\text{as }z\to0.
\end{equation}
This expansion shows that, at the threshold, the squared distance
to the fixed-scale extremal subspace increases quartically as the
scale moves away from its minimizing value.
The quartic behavior here concerns minimization over the scale
for a fixed function, whereas \cite{FP24} establishes a quartic
lower bound for the deficit in terms of the distance to the
set of optimizers.
\end{remark}

We now give the proof of Theorem~\ref{thm:unified}.

\begin{proof}[Proof of Theorem~\ref{thm:unified}]
By Lemma~\ref{sharp:conjugation}, it suffices to prove the theorem
in $\mathcal A$.
The map $T$ preserves the energy, the weighted norm, the deficit,
the distance and the balance parameter $\lambda_2=\sqrt{\frac{D_b(u)}{I_a(u)}}$, and maps the extremal functions in $\mathcal A$ to those in $\mathcal B$.

Let $u\ne0$ be in $\mathcal C^2_{K,a,b}$.
Then $I_a(u)>0$.
We also have $D_b(u)>0$. Otherwise, $u$ is constant along almost every radial ray,
and the finiteness of $I_a(u)$ guarantees the constant to vanish, since
$\int_0^\infty r^{N-1-2a}dr=\infty$.
Thus $\lambda_2=\lambda_2(u)>0$.
When $u\in C_c^\infty(\mathbb R^N\setminus\{o\})\setminus\{0\}$, it follows from the integration by parts that
\begin{equation}\label{eq:weighted-cross-identity}
2\int_{\RN}u\mathcal R_K(u)r^{-a-b}dx
=-\bigl(N-a-b-1\bigr)I_{a,b}(u).
\end{equation}
Then we obtain
\begin{equation}\label{sharp:remainder}
\delta_{K,a,b}(u)=\frac1{2\lambda_2}
\int_{\RN}\bigl|\mathcal R_K(u)+\operatorname{sgn}(b+1-a)
\lambda_2 r^{b-a}u\bigr|^2r^{-2b}dx.
\end{equation}
By the approximation argument, \eqref{sharp:embedding}, and
continuity of $\lambda_2$ at nonzero functions,
we can extend \eqref{sharp:remainder}
to every nonzero $u\in \mathcal C^2_{K,a,b}$.

Let $u(r,\sigma)=e^{-\frac{s}{2}}h(s,\sigma)$ with
$s=\frac{2\lambda}{|b+1-a|} r^{b+1-a}$ and $\lambda=\lambda_2$.
By the change of variables, it follows from the definition of $I_{a,b}$ and
\eqref{sharp:remainder} that
\begin{align}
I_{a,b}(u)&=\frac{\Gamma(\nu)}{|b+1-a|}\left(\frac{2\lambda}{|b+1-a|}\right)^{-\nu}\int_{\partial K}\int_0^\infty
|h|^2d\mu_\nu d\mu_K,\label{eq4.13}\\
\delta_{K,a,b}(u)&=\Gamma(\nu)\left(\frac{2\lambda}{|b+1-a|}\right)^{-\nu}
\int_{\partial K}\int_0^\infty
s|\partial_sh|^2d\mu_\nu d\mu_K.\label{eq4.14}
\end{align}
The finiteness of these integrals follows from
\eqref{sharp:embedding} and \eqref{sharp:remainder}.
Then we have $h(\cdot,\sigma)\in \mathcal D_\nu=\big\{
h\in W_{\mathrm{loc}}^{1,2}(0,\infty):\int_0^\infty
\big(|h(s)|^2+s|h'(s)|^2\big)
s^{\nu-1}e^{-s}ds<\infty\big\}$
for almost every $\sigma$.

For brevity, we write $d_\lambda(u)
=\inf_{\Phi\in L^2(\partial K,d\mu_K)}
I_{a,b}\bigl(u-\Phi g_\lambda\bigr)^{\frac{1}{2}}$
when $\lambda=\lambda_2$.
By Lemma~\ref{l5}, we have
\begin{equation}\label{eq4.15}
d_\lambda(u)^2=\frac{\Gamma(\nu)}{|b+1-a|}\left(\frac{2\lambda}{|b+1-a|}\right)^{-\nu}\int_{\partial K}
\inf_{q\in\mathbb R}\int_0^\infty
|h(s,\sigma)-q|^2d\mu_\nu(s)d\mu_K(\sigma).
\end{equation}
This, together with Lemma~\ref{l2}, implies that
\begin{equation}\label{sharp:chain}
\delta_{K,a,b}(u)\ge|b+1-a|d_\lambda(u)^2
\ge|b+1-a|d_{K,a,b}(u,\mathcal M_{K,a,b})^2.
\end{equation}

We next characterize the equality conditions.
Equality in the first inequality of \eqref{sharp:chain}
holds if and only if, for almost every $\sigma$,
\[h(s,\sigma)=\Phi_0(\sigma)+\Phi_1(\sigma)\bigl(s-\nu\bigr).\]
Since the functions
$1$ and $s-\nu$ are orthogonal in $L^2(\mu_\nu)$ with squared
norms $1$ and $\nu$, and
\eqref{eq4.13} is finite, we have $\Phi_0,\Phi_1 \in L^2(d\mu_K)$.
By the fact that
$\int s\,d\mu_\nu=\nu$ and $\int s\bigl(s-\nu\bigr)d\mu_\nu=\nu,$
one has from \eqref{eq4.4} and \eqref{eq4.5} that
\begin{equation}\label{sharp:balance-orthogonality}
D_b(u)-\lambda^2I_a(u)
=-2\lambda \nu \Gamma(\nu) \left(\frac{2\lambda}{|b+1-a|}\right)^{-\nu}\langle\Phi_0,\Phi_1\rangle
\end{equation}
for every $u=e^{-\frac{s}{2}}h(s,\sigma)$ with $h(s,\sigma)=\Phi_0(\sigma)+\Phi_1(\sigma)\bigl(s-\nu\bigr)$.
From $\lambda=\lambda_2(u)$, the left-hand side vanishes, and then
$\langle\Phi_0,\Phi_1\rangle=0$.
Therefore, by Proposition~\ref{prop:nearest-scales}, $\lambda$ is a nearest
scale if and only if
\[
\|\Phi_0\|_2^2\ge\nu\|\Phi_1\|_2^2.
\]
This is just the equality case of the second inequality of \eqref{sharp:chain}. The proof of sufficiency is straightforward, and we omit the argument.

We now prove sharpness on $C_c^\infty(\RN\setminus\{o\})$.
We may choose a nonzero bounded $\Psi\in L^2(d\mu_K)$ with
$\int\Psi d\mu_K=0$, since $\mu_K$ is not a point mass for
$N\ge2$.
Take $\Phi_0=1$ and $\Phi_1=\varepsilon\Psi$, with
$\langle\Phi_0,\Phi_1\rangle=0$ and $\varepsilon>0$ small enough such that
$\mu_K(\partial K)\ge\nu\varepsilon^2\|\Psi\|_2^2$.
For $u(r,\sigma)=e^{-\frac{s}{2}}(1+\varepsilon\Psi(s-\nu))$ with
$s=\frac{2\lambda}{|b+1-a|} r^{b+1-a}$ for a given $\lambda>0$,
it follows from \eqref{eq4.15} and the equality conditions that
\begin{align*}
d_{K,a,b}(u,\mathcal M_{K,a,b})^2
&=\frac{\Gamma(\nu)}{|b+1-a|}\left(\frac{2\lambda}{|b+1-a|}\right)^{-\nu}\nu\varepsilon^2\|\Psi\|_2^2>0,\\
\delta_{K,a,b}(u)&=|b+1-a|d_{K,a,b}(u,\mathcal M_{K,a,b})^2.
\end{align*}
This shows the sharpness on $\mathcal C^2_{K,a,b}$.
By Lemma~\ref{sharp:density}, choose compactly supported smooth functions
$u_j\rightarrow u$
in the energy norm as $j\to\infty$.
Energy convergence and the continuous embedding imply that
$$\delta_{K,a,b}(u_j)\to\delta_{K,a,b}(u).$$
By \eqref{sharp:embedding}, one also has
$I_{a,b}(u_j-u)\rightarrow0$ as $j\to\infty.$
Since distance is Lipschitz with constant $1$ in
$L^2(\RN,r^{-a-b-1}dx)$, one has
\[
d_{K,a,b}(u_j,\mathcal M_{K,a,b})
\rightarrow d_{K,a,b}(u,\mathcal M_{K,a,b})
\quad\text{as }j\to\infty.
\]
Thus there exists $u_j\in C_c^\infty(\RN\setminus\{o\})$
such that the corresponding ratios converge to $|b+1-a|$ as $j\to\infty$.
Therefore, the constant $|b+1-a|$ is sharp on
$C_c^\infty(\RN\setminus\{o\})$.

Assume that $u\in C_c^\infty(\RN\setminus\{o\})$ attains equality.
Since $u$ has compact support, $u(r,\sigma)=0$ for all
sufficiently large $r$.
Since $g_\lambda(r)>0$ and $\bigl(b+1-a\bigr)\ne0$,
it follows from \eqref{eq:unified-equality} that
$\Phi_0=\Phi_1=0$ almost everywhere,
and $u=0$.
\end{proof}

\subsection{Sharp $L^2$ Heisenberg stability}

We now establish the sharp stability estimates
in our anisotropic radial setting.
We refer to \cite{LLLS} for the stability of the Heisenberg uncertainty principle on half-spaces and orthants.

\begin{theorem}\label{thm:H-L2}
Let $K\in\Ksc$.
\begin{enumerate}[label=\textup{(\roman*)}]
\item
For every $u\in\mathcal H_K$,
\begin{equation}\label{eHnonsharp}
\delta_{K,H,non}(u)
\geq 4 \inf_{v\in\mathcal M_{K,H,non}}\|u-v\|^2_{L^2(\mathbb R^N)},
\end{equation}
with equality if and only if
\begin{equation}\label{sharp:Hnon-equality}
u(r,\sigma)=e^{-\frac{r^2}{2}}
\biggl[\Phi_0(\sigma)+\Phi_1(\sigma)\biggl(r^2-\frac N2\biggr)\biggr],
\quad \Phi_0,\Phi_1\in L^2(\partial K,d\mu_K).
\end{equation}
The constant $4$ is sharp, including on
$C_c^\infty(\RN\setminus\{o\})$.
If equality holds for $u\in C_c^\infty(\RN\setminus\{o\})$, then $u=0$.
\item
For every $u\in\mathcal H_K\setminus\{0\},$
\begin{equation}\label{eHsharp}
\delta_{K,H}(u)
\geq 2 \inf_{v\in\mathcal M_{K,H}}\|u-v\|^2_{L^2(\mathbb R^N)},
\end{equation}
with equality if and only if
\[
u(r,\sigma)=e^{-\frac{\lambda r^2}{2}}
\bigg[\Phi_0(\sigma)+\Phi_1(\sigma)\bigg(\lambda r^2-\frac N2\bigg)\bigg],
\]
where $\lambda=\lambda_H(u):=
\bigg(\frac{\int_{\RN}|\mathcal R_K(u)|^2dx}{\int_{\RN}r^2|u|^2dx}\bigg)^{\frac12},$
$\Phi_0,\Phi_1\in L^2(\partial K,d\mu_K)$
are not both zero and satisfy
$\langle\Phi_0,\Phi_1\rangle=0$ and $\|\Phi_0\|_2^2\ge\frac N2\|\Phi_1\|_2^2.$

The stability constant $2$ is sharp, including on $C_c^\infty(\RN\setminus\{o\})$.
If equality holds for $u\in C_c^\infty(\RN\setminus\{o\})$, then $u=0$.
\end{enumerate}
\end{theorem}

We note that there is no orthogonality or norm condition on $\Phi_0,\Phi_1$ in \eqref{sharp:Hnon-equality}.

\begin{proof}
We first prove part \textup{(i)}.
By \eqref{eq:weighted-cross-identity} with $(a,b)=(-1,0)$
and the definition of $\delta_{K,H,non}$, one has
\[
\delta_{K,H,non}(u)=\int_{\RN}\bigl|\mathcal R_K(u)+ru\bigr|^2dx.
\]
Let $\nu_H=\frac N2$, $s=r^2$, and
$h(s,\sigma)=e^{\frac{r^2}{2}}u(r,\sigma)$.
By the change of variables $s=r^2$ and Lemma \ref{l5}, we deduce
\begin{align*}
\delta_{K,H,non}(u)
&=2\Gamma(\nu_H)
\int_{\partial K}\int_0^\infty
s|\partial_s h|^2d\mu_{\nu_H}d\mu_K,
\\
\inf_{v\in\mathcal M_{K,H,non}}\|u-v\|^2_{L^2(\mathbb R^N)}
&=\frac{\Gamma(\nu_H)}2\int_{\partial K}\inf_{q\in\mathbb R}
\int_0^\infty|h-q|^2d\mu_{\nu_H}d\mu_K.
\end{align*}
By these identities and the embedding
$\mathcal H_K\hookrightarrow L^2(\RN)$,
it follows from Lemma~\ref{l2} for $\mu_K$-almost every $\sigma$ that
\[
\delta_{K,H,non}(u)
\geq 4 \inf_{v\in\mathcal M_{K,H,non}}\|u-v\|^2_{L^2(\mathbb R^N)}.
\]
The equality characterization in Lemma~\ref{l2} shows that
equality holds if and only if
\[
h(s,\sigma)
=\Phi_0(\sigma)+\Phi_1(\sigma)(s-\nu_H),
\]
which is precisely \eqref{sharp:Hnon-equality}.

By Lemma~\ref{sharp:density} with $(a,b)=(-1,0)$
and $\lambda=1$, every function in
\eqref{sharp:Hnon-equality} belongs to $\mathcal H_K$.
The function $u(r,\sigma)=r^2e^{-\frac{r^2}{2}}$
attains equality and has positive distance from
$\mathcal M_{K,H,non}$.
Hence, the constant $4$ is sharp on $\mathcal H_K$.
To prove sharpness on $C^\infty_c(\mathbb R^N\setminus \{o\})$, choose
$u_j\in C_c^\infty(\RN\setminus\{o\})$
converging to $u$ in $\mathcal H_K$.
The embedding $\mathcal H_K\hookrightarrow L^2(\RN)$
and continuity of distance imply that
\[
\frac{\delta_{K,H,non}(u_j)}
{\inf_{v\in\mathcal M_{K,H,non}}\|u_j-v\|^2_{L^2(\mathbb R^N)}}
\rightarrow 4, \quad\text{as }j\to\infty.
\]
Note that the denominators are positive for sufficiently large $j$,
since the limiting distance is positive.
Moreover, if $u\in C^\infty_c(\mathbb R^N\setminus \{o\})$ attains equality,
then \eqref{sharp:Hnon-equality} and compact support force
$\Phi_0=\Phi_1=0$ almost everywhere.
Thus only zero attains equality on the test class.

For part \textup{(ii)}, \eqref{eHsharp} with sharp constant $2$ is an immediate consequence of Theorem~\ref{thm:unified} with $(a,b)=(-1,0)$.
The equality characterization with $\lambda=\lambda_H(u)$ follows from
\eqref{eq:unified-equality} and
\eqref{eq:unified-conditions}.
The sharpness on the test class and nonattainment by nonzero compactly supported smooth functions also follow from Theorem~\ref{thm:unified}.
\end{proof}

\begin{remark}
Let $K$ be the Euclidean unit ball, $\lambda>0$, and
$0<|\varepsilon|\le\sqrt 2$. Then
\[
u_\varepsilon(x)=e^{-\frac{\lambda|x|^2}{2}}
\bigg[1+\varepsilon\frac{x_1}{|x|}
\bigg(\lambda|x|^2-\frac N2\bigg)\bigg]
\]
belongs to $\mathcal H_K$ and is an equality case in
\eqref{eHsharp}. Indeed, the angular coefficients are
$\Phi_0=1$ and $\Phi_1=\varepsilon \sigma\cdot e_1$,
where $e_1$ denotes the first standard basis vector.
Then we have the identities
$\int_{\mathbb S^{N-1}}\sigma\cdot e_1\,d\sigma=0$ and
$\int_{\mathbb S^{N-1}}|\sigma\cdot e_1|^2\,d\sigma
=\frac{|\mathbb S^{N-1}|}{N}$.
As $|\varepsilon|\le\sqrt2$, Proposition~\ref{prop:nearest-scales}
shows that $\lambda$ is a nearest scale. For each $\sigma$, one has
$q=\int h\,d\mu_{N/2}=1$ in \eqref{eq4.15}.
Therefore, \eqref{eq4.14} and \eqref{eq4.15} give
\[
\inf_{v\in\mathcal M_{K,H}}\|u_\varepsilon-v\|^2_{L^2(\mathbb R^N)}
=\frac{\varepsilon^2}{2}\Bigl(\frac\pi\lambda\Bigr)^{\frac{N}{2}},
\quad
\delta_{K,H}(u_\varepsilon)
=\varepsilon^2\Bigl(\frac\pi\lambda\Bigr)^{\frac{N}{2}}>0.
\]
For $|\varepsilon|>\sqrt 2$, the same family has the two nearest
scales described in Proposition~\ref{prop:nearest-scales}.

We note that nonconstant angular coefficients can also yield equality in
\eqref{eHsharp}.
Indeed, if both angular coefficients in the scale-invariant
equality formula are constants, their orthogonality and the norm condition imply $\Phi_1=0$.
\end{remark}

\section{The \texorpdfstring{$L^p$}{Lp} stability of anisotropic radial inequalities}\label{s-lp}

In this section, we prove Theorems~\ref{tHcompletion} and~\ref{tA}.
We begin with the stability of the $L^p$ CKN inequalities in the parameter region $\mathcal A_p$, and then treat the $L^p$ Heisenberg stability without the parameter restriction.
In both cases, we first establish the estimates for
compactly supported smooth functions and then extend them
to the corresponding completion spaces.
The preliminary inequalities hold for $1<p<\infty$,
while the stability estimates require $p\ge2$.
Throughout, $p'=\frac{p}{p-1}$ and $\nu_H=\frac{N}{p'}$.

\subsection{CKN type inequalities in the region \texorpdfstring{$\mathcal A_p$}{Ap}}\label{s4}\label{s4Ap}

We now establish weighted $L^p$ stability
for the anisotropic radial CKN inequalities
in the region $\mathcal A_p$.
We refer the reader to \cite{LLLSCKN} for the stability of
sharp $L^2$-CKN inequalities by using
related weighted Poincar\'e inequalities on half-spaces and orthants.

Recall the parameter region $\mathcal A_p$ and
the constant $\nu_p$ defined in Section~\ref{s1}.
We first establish the following one-dimensional stability estimate.

\begin{lemma}\label{l3}
Let $1< p<\infty$, $(a,b)\in\mathcal A_p$.
Then, for $f\in C_c^\infty(0,\infty)$,
\begin{align*}
&\int_0^\infty
\big|f'(r)+\operatorname{sgn}\bigl(b+1-a\bigr)\lambda r^{b-a}f(r)\big|^p
r^{N-1-bp}dr\\
&\geq
\frac{(p\lambda)^{p-1}|b+1-a|}{P_{\nu_p,p}}
\inf_{c\in\mathbb R}\int_0^\infty\bigl|f(r)-cF_\lambda(r)\bigr|^pr^{N-(p-1)a-b-2}dr.
\end{align*}
where $F_\lambda(r)=\exp\left(-\frac{\lambda}{|b+1-a|}r^{b+1-a}\right)\ \mathrm{for}\ \lambda>0.$
\end{lemma}

\begin{proof}
We define $g$ by $f=F_\lambda g$.
Since $F_\lambda'+\operatorname{sgn}(b+1-a)\lambda r^{b-a}F_\lambda=0$,
\begin{equation}\label{eq5.5}
f'+\operatorname{sgn}(b+1-a)\lambda r^{b-a}f=F_\lambda g',
\quad f-cF_\lambda=F_\lambda\bigl(g-c\bigr).
\end{equation}
Let $s=\frac{p\lambda}{|b+1-a|}r^{b+1-a},$ and $h(s)=g(r)$.
We have $F_\lambda^p=e^{-s}$, $h\in C_c^\infty(0,\infty)$.
By the change of variables $s=\frac{p\lambda}{|b+1-a|}r^{b+1-a}$,
it implies
\begin{align*}
\int_0^\infty|g'|^pF_\lambda^p r^{N-1-bp}dr
&=(p\lambda)^{p-1}
\biggl(\frac{p\lambda}{|b+1-a|}\biggr)^{-\nu_p}
\int_0^\infty s|h'|^p s^{\nu_p-1}e^{-s}ds,\\
\int_0^\infty|g(r)-c|^pF_\lambda^p r^{N-(p-1)a-b-2}dr
&=\frac1{|b+1-a|}
\biggl(\frac{p\lambda}{|b+1-a|}\biggr)^{-\nu_p}
\int_0^\infty|h(s)-c|^p s^{\nu_p-1}e^{-s}ds.
\end{align*}
Since $(a,b)\in\mathcal A_p$, we have $\nu_p\ge p-1>0$.
From Lemma~\ref{l2} and \eqref{eq5.5}, it shows the desired result.
\end{proof}

For a nonzero test function $u$, define
\[
\lambda_p(u)=
\biggl(\frac{\int_{\mathbb R^N}|\mathcal R_K(u)|^p \|x\|_K^{-bp}dx}{\int_{\mathbb R^N}|u|^p \|x\|_K^{-ap}dx}
\biggr)^{\frac{1}{p}}.
\]

\begin{theorem}\label{tAcompletion}
Let $2\le p<\infty$, $K\in\Ksc$ and
$(a,b)\in\mathcal A_p$. Then, for every
$u\in C_c^\infty(\mathbb R^N\setminus\{o\})\setminus\{0\},$
one has
\begin{align}\label{eAcompletion}
\delta_{p,K,\mathcal A_p}(u)
\geq\frac{\kappa_p p^{p-2}|b+1-a|}{P_{\nu_p,p}}
d_{p,K,\mathcal A_p}
\big(u,\mathcal M_{p,K,\mathcal A_p}\big)^p.
\end{align}
\end{theorem}

\begin{proof}
For every
$u\in C_c^\infty(\mathbb R^N\setminus\{o\})\setminus\{0\}$, the next integrals are finite:
\[
D_{p,b}(u)\!=\!\!\int_{\mathbb R^N}\!\bigl|\mathcal R_K(u)\bigr|^p \|x\|_K^{-bp}dx,\
I_{p,a}(u)\!=\!\!\int_{\mathbb R^N}\!|u|^p \|x\|_K^{-ap}dx,\
I_{p,a,b}(u)\!=\!\!\int_{\mathbb R^N}\!|u|^p \|x\|_K^{-(p-1)a-b-1}dx.
\]
Since $u\ne0$, one has $I_{p,a}>0$.
If $D_{p,b}(u)=0$, $u$ is constant on almost every radial ray and compact support
forces $u=0$, a contradiction.
Hence, $\lambda_p(u)\in(0,\infty).$
Thus $\lambda=\lambda_p(u)=\Bigl(\frac{D_{p,b}(u)}{I_{p,a}(u)}\Bigr)^{\frac{1}{p}}>0$, and
\[\lambda^{-(p-1)}D_{p,b}(u)=\lambda I_{p,a}(u)=D_{p,b}(u)^{\frac{1}{p}}I_{p,a}(u)^{\frac{1}{p'}}.\]
By the homogeneity of $\|\cdot\|_K$, it is clear that $x\cdot\nabla\|x\|_K=\|x\|_K$, for every $x\neq o$.
By the integration by parts, it gives
\begin{align*}
p\int_{\mathbb R^N}
\frac{|u|^{p-2}u\mathcal R_K(u)}
{\|x\|_K^{(p-1)a+b}}dx
&=\int_{\mathbb R^N}
\frac{x\cdot \nabla |u|^{p}}
{\|x\|_K^{(p-1)a+b+1}}dx\\
&=-\big(N-1-(p-1)a-b\big)
\int_{\mathbb R^N}
\frac{|u|^p}
{\|x\|_K^{(p-1)a+b+1}}dx.
\end{align*}
Let $k=N-1-(p-1)a-b$.
Since $(a,b)\in\mathcal A_p$ and $\nu_p>0$,
$\operatorname{sgn}(b+1-a)\cdot k=|k|.$
Therefore, by the definition of $R_p$,
\begin{align*}
&\frac{1}{p}\int_{\mathbb R^N}\frac{1}{\|x\|_K^{bp}}
R_p\bigg(-\operatorname{sgn}(b+1-a)\lambda_p(u)^{\frac1p}u\|x\|_K^{b-a},\lambda_p(u)^{-\frac{p-1}{p}}\mathcal R_K(u)\bigg)dx\\
&=
\bigg(\int_{\mathbb R^N}\frac{|\mathcal R_K(u)|^p}{\|x\|_K^{bp}}dx\bigg)^{\frac1p}
\bigg(\int_{\mathbb R^N}\frac{|u|^p}{\|x\|_K^{ap}}dx\bigg)^{\frac{p-1}{p}}
-\bigg|\frac{N-1-(p-1)a-b}{p}\bigg|\int_{\mathbb R^N}\frac{|u|^p}{\|x\|_K^{(p-1)a+b+1}}dx\\
&=\delta_{p,K,\mathcal A_p}(u).
\end{align*}
By Lemma \ref{lUC}, one has
\begin{align*}
\delta_{p,K,\mathcal A_p}(u)
&\geq
\frac{\kappa_p}{p}\lambda_p(u)^{-(p-1)}
\int_{\mathbb R^N}
\frac{\big|\mathcal R_K(u)+\operatorname{sgn}(b+1-a)\lambda_p(u)\|x\|_K^{b-a}u\big|^p}
{\|x\|_K^{bp}}dx.
\end{align*}
Using the anisotropic polar formula and $\mathcal R_K(u)=\partial_r u$, one has
\begin{align*}
&\int_{\mathbb R^N}
\frac{\big|\mathcal R_K(u)+\operatorname{sgn}(b+1-a)\lambda_p(u)\|x\|_K^{b-a}u\big|^p}{\|x\|_K^{bp}}dx\\
&=\int_{\partial K}\int_0^\infty
\big|\partial_r u(r,\sigma)+\operatorname{sgn}(b+1-a)\lambda_p(u)r^{b-a}u(r,\sigma)\big|^p
r^{N-1-bp}drd\mu_K(\sigma).
\end{align*}
Therefore, one gets from Lemma \ref{l3} that
\begin{align*}
\delta_{p,K,\mathcal A_p}(u)
\ge\frac{\kappa_p p^{p-2}|b+1-a|}{P_{\nu_p,p}}
\int_{\partial K}\inf_{c\in\mathbb R}\int_0^\infty
\bigl|u(r,\sigma)-cF_{\lambda_p(u)}(r)\bigr|^p
r^{N-(p-1)a-b-2}drd\mu_K(\sigma),
\end{align*}
where $F_\lambda(r)=\exp\left(-\frac{\lambda}{|b+1-a|}r^{b+1-a}\right)\ \mathrm{for}\ \lambda>0.$
Note that for $\nu_p>0$,
\begin{align*}
\int_0^\infty
F_{\lambda_p(u)}(r)^p
r^{N-(p-1)a-b-2}dr
&=\frac{1}{|b+1-a|}
\bigg(\frac{p\lambda_p(u)}{|b+1-a|}\bigg)^{-\nu_p}\Gamma(\nu_p)\in(0,\infty).
\end{align*}
For $u\in C_c^\infty(\mathbb R^N\setminus\{o\}),$
one has
\begin{align*}
\int_{\mathbb R^N}\frac{|u|^p}{\|x\|_K^{(p-1)a+b+1}}dx<\infty.
\end{align*}
By Lemma \ref{lApembedding}, this finiteness holds for $u\in\mathcal C_{K,a,b}^p$.
Thus, both hypotheses of Lemma \ref{l5}
are satisfied.
Applying Lemma \ref{l5}, we obtain
\begin{align*}
\delta_{p,K,\mathcal A_p}(u)\geq\frac{\kappa_p p^{p-2}|b+1-a|}{P_{\nu_p,p}}
\inf_{\Phi\in L^p(\partial K,d\mu_K)}
\int_{\mathbb R^N}
\frac{\big|u(x)-\Phi\big(\sigma_K(x)\big)F_{\lambda_p(u)}(\|x\|_K)\big|^p}
{\|x\|_K^{(p-1)a+b+1}}dx.
\end{align*}
This is the desired result.
\end{proof}

To extend the result to $\mathcal C^p_{K,a,b}$,
we prove the following embedding lemma.

\begin{lemma}\label{lApembedding}
Let $1<p<\infty$, $K\in\Ksc$ and
$(a,b)\in\mathcal A_p$. Then $\mathcal C_{K,a,b}^p$ is continuously embedded into
$L^p\Big(\mathbb R^N,\frac{dx}{\|x\|_K^{(p-1)a+b+1}}\Big).$
More precisely,
\begin{align}\label{eApembedding}
\int_{\mathbb R^N}
\frac{|u|^p}
{\|x\|_K^{(p-1)a+b+1}}dx
\leq \frac{1}{|N-1-(p-1)a-b|}
\bigg[\int_{\mathbb R^N}\frac{|\mathcal R_K(u)|^p}{\|x\|_K^{bp}}dx
+\bigl(p-1\bigr)\int_{\mathbb R^N}\frac{|u|^p}{\|x\|_K^{ap}}dx\bigg]
\end{align}
for every $u\in\mathcal C_{K,a,b}^p$.
\end{lemma}

\begin{proof}
For $u\in C_c^\infty(\mathbb R^N\setminus\{o\}),$
the anisotropic $L^p$-CKN inequality gives
\begin{align*}
\frac{|N-1-(p-1)a-b|}{p}
\int_{\mathbb R^N}\frac{|u|^p}{\|x\|_K^{(p-1)a+b+1}}dx
\leq
\bigg(\int_{\mathbb R^N}\frac{|\mathcal R_K(u)|^p}{\|x\|_K^{bp}}dx\bigg)^{\frac1p}
\bigg(\int_{\mathbb R^N}\frac{|u|^p}{\|x\|_K^{ap}}dx\bigg)^{\frac{p-1}{p}}.
\end{align*}
Young's inequality implies \eqref{eApembedding} for $u\in C_c^\infty(\mathbb R^N\setminus\{o\})$.
Since $N-1-(p-1)a-b\neq0$ for $(a,b)\in\mathcal A_p$,
\eqref{eApembedding} holds for all $u\in\mathcal C_{K,a,b}^p$ by the approximation argument.
\end{proof}

We are now in the position to show Theorem \ref{tA}.

\begin{proof}[Proof of Theorem \ref{tA}]
Let $u\in\mathcal C^p_{K,a,b}\setminus\{0\}$.
Choose $u_j\in C_c^\infty(\RN\setminus\{o\})$
such that $u_j\to u$ in $\mathcal C^p_{K,a,b}$ as $j\to\infty$.
The two integrals in the norm of $\mathcal C^p_{K,a,b}$
converge as $j\to\infty$.
Then Lemma~\ref{lApembedding} implies the convergence $u_j\to u$ in
$L^p(\RN,\hK{x}^{-(p-1)a-b-1}dx)$ as $j\to\infty$.
Therefore, $\delta_{p,K,\mathcal A_p}(u_j)\to \delta_{p,K,\mathcal A_p}(u)$ as $j\to\infty$.
We also have $|d_{p,K,\mathcal A_p}\bigl(u_j,\mathcal M_{p,K,\mathcal A_p}\bigr)-d_{p,K,\mathcal A_p}\bigl(u,\mathcal M_{p,K,\mathcal A_p}\bigr)|\leq\|u_j-u\|_{L^p(\mathbb R^n,{\|x\|_K^{-(p-1)a-b-1}}dx)}\to0$ as $j\to\infty$.
For sufficiently large $j$, we have $u_j\ne0$.
Therefore, we get Theorem~\ref{tA}
by applying Theorem~\ref{tAcompletion} to $u_j$
and letting $j\to\infty$.
\end{proof}

\subsection{Heisenberg type inequalities}\label{s3}

This subsection uses the weighted $L^p$-Poincar\'e inequality
for Gamma measures to obtain the anisotropic radial Heisenberg stability estimates.
See, e.g., \cite{DGLLHyperbolic} for the Heisenberg stability
by the classical weighted Poincar\'e inequalities.

The following estimate follows from a change of variables
and Lemma~\ref{l2}.

\begin{lemma}
\label{lH}
Let $1<p<\infty$ and $\lambda>0$. Then, for
$f\in C_c^\infty(0,\infty)$,
\begin{align}\label{eH1}
\int_0^\infty\big|f'(r)+\lambda r^{\frac{1}{p-1}}f(r)\big|^pr^{N-1}dr
\ge\frac{p'(p\lambda)^{p-1}}{P_{\nu_H,p}}
\inf_{c\in\mathbb R}\int_0^\infty\bigl|f(r)-cH_\lambda(r)\bigr|^pr^{N-1}dr,
\end{align}
where $\nu_H=\frac{N}{p'}$, and $H_\lambda(r)=\exp\big(-\frac{\lambda}{p'}r^{p'}\big)$ for $r>0$.
\end{lemma}

\begin{proof}
For $f\in C_c^\infty(0,\infty)$, we define $g$ by $f=H_\lambda g$.
Since $H_\lambda'+\lambda r^{\frac{1}{p-1}}H_\lambda=0$,
we have, for every $c\in\mathbb R$,
\begin{equation}\label{eq5.2}
f'+\lambda r^{\frac{1}{p-1}}f=H_\lambda g',\quad
f-cH_\lambda=H_\lambda\bigl(g-c\bigr).
\end{equation}
We consider the function $h(s):=g(r)$ with
$s=\bigl(\frac{p\lambda}{p'}\bigr)r^{p'}$.
Then $H_\lambda^p=e^{-s}$ and $h\in C_c^\infty(0,\infty)$.
By the change of variables $s=\bigl(\frac{p\lambda}{p'}\bigr)r^{p'}$, one has
\begin{align*}
\int_0^\infty|g'|^pH_\lambda^p r^{N-1}dr
&=(p\lambda)^{p-1}
\biggl(\frac{p\lambda}{p'}\biggr)^{-\nu_H}
\int_0^\infty s|h'|^p s^{\nu_H-1}e^{-s}ds,\\
\int_0^\infty|g-c|^pH_\lambda^p r^{N-1}dr
&=\frac1{p'}\biggl(\frac{p\lambda}{p'}\biggr)^{-\nu_H}
\int_0^\infty|h-c|^p s^{\nu_H-1}e^{-s}ds.
\end{align*}
By $\nu_H=\frac{N}{p'}>0$, Lemma~\ref{l2} implies
\begin{align*}
\int_0^\infty|g'|^pH_\lambda^p r^{N-1}dr
\geq \frac{p'(p\lambda)^{p-1}}{P_{\nu_H,p}}
\inf_{c\in\mathbb R}\int_0^\infty|g-c|^pH_\lambda^p r^{N-1}dr,
\end{align*}
This, together with \eqref{eq5.2}, implies \eqref{eH1}.
\end{proof}

Using this lemma, we first give the $L^p$ Heisenberg stability estimates
on $C_c^\infty(\RN\setminus\{o\})$,
and then extend them to the class $\mathcal H_K^p$.

\begin{theorem}\label{t2}
Let $2\le p<\infty$ and $K\in\Ksc$.
\begin{enumerate}[label=\textup{(\roman*)}]
\item
For every $u\in C_c^\infty(\RN\setminus\{o\})$, one has
\begin{equation}\label{eHn}
\delta_{p,K,H,non}(u)
\ge
\frac{\kappa_p p'p^{p-1}}{P_{\nu_H,p}}
\inf_{v\in\mathcal M_{p,K,H,non}}\|u-v\|^p_{L^p(\mathbb R^N)}.
\end{equation}
\item
For every
$u\in C_c^\infty(\RN\setminus\{o\})\setminus\{0\}$,
one has
\begin{equation}\label{eHs}
\delta_{p,K,H}(u)
\geq
\frac{\kappa_p p'p^{p-2}}{P_{\nu_H,p}}
\inf_{v\in\mathcal M_{p,K,H}}\|u-v\|^p_{L^p(\mathbb R^N)}.
\end{equation}
\end{enumerate}
\end{theorem}

\begin{proof}
We first prove part \textup{(i)}.
By the divergence theorem,
\begin{align*}
N\int_{\mathbb R^N}|u|^pdx
=
-p\int_{\mathbb R^N}
|u|^{p-2}u\,x\cdot\nabla u dx.
\end{align*}
Therefore, it follows from $x\cdot\nabla u=\|x\|_K\mathcal R_K(u)$,
the definition of $R_p$, Lemma~\ref{lUC}, and anisotropic polar coordinates that
\begin{align*}
\delta_{p,K,H,non}(u)&=\int_{\mathbb R^N}
\bigl|\mathcal R_K(u)\bigr|^pdx
+\bigl(p-1\bigr)\int_{\mathbb R^N}\|x\|_K^{p'}|u|^pdx+p\int_{\mathbb R^N}|u|^{p-2}u\,x\cdot\nabla udx\\
&=\int_{\mathbb R^N}
R_p\bigg(-u\|x\|_K^{\frac{1}{p-1}},\mathcal R_K(u)\bigg)dx\\
&\ge
\kappa_p
\int_{\mathbb R^N}
\bigg|\mathcal R_K(u)+\|x\|_K^{\frac{1}{p-1}}u\bigg|^pdx\\
&=\kappa_p\int_{\partial K}\int_0^\infty\bigg|\partial_ru(r,\sigma)
+r^{\frac{1}{p-1}}u(r,\sigma)\bigg|^p
r^{N-1}drd\mu_K(\sigma).
\end{align*}
Applying Lemma \ref{lH} with $\lambda=1$ and
$f(r)=u(r,\sigma)$ for every fixed $\sigma\in\partial K$, it gives
\begin{align*}
\int_0^\infty\big|\partial_ru(r,\sigma)+r^{\frac{1}{p-1}}u(r,\sigma)\big|^pr^{N-1}dr
\ge \frac{p'p^{p-1}}{P_{\nu_H,p}}\inf_{c\in\mathbb R}
\int_0^\infty\big|u(r,\sigma)-cH_1(r)\big|^pr^{N-1}dr.
\end{align*}
By the definition of $H_1$ and $\nu_H>0$, it is clear that
\begin{align*}
\int_0^\infty
H_1(r)^p r^{N-1}dr
=\frac{1}{p'}\Big(\frac{p}{p'}\Big)^{-\nu_H}\Gamma(\nu_H)
\in(0,\infty).
\end{align*}
Then, the assumptions of Lemma \ref{l5} are satisfied.
By Lemma \ref{l5} with $v(r)=H_1(r)$ and $w(r)=r^{N-1},$ we obtain
\begin{align*}
&\int_{\partial K}\inf_{c\in\mathbb R}\int_0^\infty
\bigl|u(r,\sigma)-cH_1(r)\bigr|^pr^{N-1}drd\mu_K(\sigma)\\
&=\inf_{\Phi\in L^p(\partial K,d\mu_K)}
\int_{\partial K}\int_0^\infty\bigl|u(r,\sigma)-\Phi(\sigma)H_1(r)\bigr|^p
r^{N-1}drd\mu_K(\sigma)\\
&=\inf_{v\in\mathcal M_{p,K,H,non}}\|u-v\|^p_{L^p(\mathbb R^N)}.
\end{align*}
Combining with the preceding estimates, this proves \eqref{eHn}.

We now prove part \textup{(ii)} by a dilation normalization.
For $u\in C_c^\infty(\RN\setminus\{o\})\setminus\{0\}$, we set
\[
\lambda_H(u)=\biggl(\frac{\int_{\RN}\bigl|\mathcal R_K(u)\bigr|^pdx}
{\int_{\RN}\hK{x}^{p'}|u|^pdx}\biggr)^{\frac{1}{p}},\quad
t=\lambda_H(u)^{-\frac{1}{p'}},\quad u_t(x)=u(tx).
\]
Both integrals in $\lambda_H(u)$ are positive.
Otherwise, a vanishing radial derivative would make $u$ constant
for almost every radial ray, then zero by compact support,
which contradicts the assumption.
By the homogeneity, it gives
\begin{align*}
\int_{\RN}\bigl|\mathcal R_K(u_t)\bigr|^pdx=t^{p-N}\int_{\RN}\bigl|\mathcal R_K(u)\bigr|^pdx\ \ \text{and}\
\int_{\RN}\hK{x}^{p'}|u_t|^pdx=t^{-N-p'}\int_{\RN}\hK{x}^{p'}|u|^pdx.
\end{align*}
Thus, $\lambda_H(u_t)=t^{p'}\lambda_H(u)=1$ and
$\delta_{p,K,H,non}(u_t)=p\delta_{p,K,H}(u_t)
=pt^{-N}\delta_{p,K,H}(u).$
On the other hand, it follows from $t=\lambda_H(u)^{-\frac{1}{p'}}$ that
\begin{align*}
\inf_{v\in\mathcal M_{p,K,H,non}}\|u_t-v\|^p_{L^p(\mathbb R^N)}
&=t^{-N}\inf_{\Phi\in L^p(\partial K,d\mu_K)}
\int_{\RN}\bigg|u(x)-\Phi(\sigma_K(x))
 e^{-\frac{\lambda_H(u)\hK{x}^{p'}}{p'}}\bigg|^pdx\\
&\ge t^{-N}\inf_{v\in\mathcal M_{p,K,H}}\|u-v\|^p_{L^p(\mathbb R^N)}.
\end{align*}
Therefore, we get \eqref{eHs} by applying the result (i) with $u_t$.
\end{proof}

We next prove Theorem \ref{tHcompletion},
which extends the above stability estimates to the space $\mathcal H_K^p$
by the density argument.

\begin{proof}[Proof of Theorem \ref{tHcompletion}]
Let $w\in C_c^\infty(\RN\setminus\{o\})$.
Let $\mathcal H_K^p=\mathcal C^p_{K,-\frac{1}{p-1},0}$
with the norm
$\|w\|_{\mathcal H_K^p}^{p}=\int_{\RN}|\mathcal R_K(w)|^pdx
+(p-1)\int_{\RN}\|x\|_K^{p'}|w|^pdx$.
From integration by parts and Young's inequality, it follows
\begin{align*}
N\int_{\RN}|w|^pdx
=-p\int_{\RN}|w|^{p-2}w\,x\cdot\nabla w dx
&\leq \int_{\RN}|\mathcal R_K(w)|^pdx
+(p-1)\int_{\RN}\|x\|_K^{p'}|w|^pdx\\
&\leq\max\{1,p-1\}\|w\|_{\mathcal H_K^p}^p.
\end{align*}
By the arbitrariness of $w$, it follows
$\mathcal H_K^p\hookrightarrow L^p(\RN).$

The result in part (i) is trivial for $u=0$.
Let $u\in\mathcal H_K^p\setminus\{0\}$, and choose
$u_j\in C_c^\infty(\RN\setminus\{o\})$ such that
$\|u_j-u\|_{\mathcal H_K^p}\to0$ as $j\to\infty.$
By $\mathcal H_K^p\hookrightarrow L^p(\RN),$ we obtain
$\|u_j-u\|_{L^p(\RN)}\to0,$ as $j\to\infty.$
Consequently, as $j\to\infty$,
\begin{align*}
\int_{\RN}|\mathcal R_K(u_j)|^pdx
\to
\int_{\RN}|\mathcal R_K(u)|^pdx,
\
\int_{\RN}r^{p'}|u_j|^pdx
\to
\int_{\RN}r^{p'}|u|^pdx,
\
\int_{\RN}|u_j|^pdx
\to
\int_{\RN}|u|^pdx.
\end{align*}
Hence, $\delta_{p,K,H,non}(u_j)\to\delta_{p,K,H,non}(u),$
and $\delta_{p,K,H}(u_j)\to\delta_{p,K,H}(u)$, as $j\to\infty$.

For
$\mathcal M\in
\{\mathcal M_{p,K,H,non},\mathcal M_{p,K,H}\}$,
the triangle inequality gives
\begin{align*}
\Big|
\inf_{v\in\mathcal M}\|u_j-v\|_{L^p(\RN)}
-
\inf_{v\in\mathcal M}\|u-v\|_{L^p(\RN)}
\Big|
\le
\|u_j-u\|_{L^p(\RN)}
\to0
\quad\text{as }j\to\infty.
\end{align*}
Therefore,
$\inf_{v\in\mathcal M}\|u_j-v\|_{L^p(\RN)}^p
\to\inf_{v\in\mathcal M}\|u-v\|_{L^p(\RN)}^p$,
as $j\to\infty.$
Since $u\neq0$, we have $u_j\ne0$ for all sufficiently
large $j$.
Applying Theorem~\ref{t2} to $u_j$ and letting
$j\to\infty$, we obtain both results.
\end{proof}

\begin{remark}\label{rem:trace}
When $p>N$, the completion $\mathcal H_K^p$ imposes a zero radial trace at the origin.
Indeed, H\"older's inequality shows that, for $u\in C_c^\infty(\RN\setminus\{o\})$,
\begin{align*}
|u(r,\sigma)|&=\left|\int_0^r\partial_su(s,\sigma)ds\right|\leq\int_0^rs^{-\frac{N-1}{p}}\left|\partial_su(s,\sigma)\right|s^{\frac{N-1}{p}}ds\\
&\leq \left(\int_0^rs^{-\frac{N-1}{p-1}}ds\right)^{\frac{p-1}{p}}
\left(\int_0^r \left|\partial_su(s,\sigma)\right|^ps^{N-1}ds\right)^{\frac{1}{p}}
\end{align*}
Then $\bigl\|u(r,\cdot)\bigr\|_{L^p(\partial K,d\mu_K)}
\leq \Bigl(\frac{p-1}{p-N}\Bigr)^{\frac{1}{p'}}
r^{1-\frac{N}{p}}\bigl\|\mathcal R_K(u)\bigr\|_{L^p(\RN)}.$
By the density argument, this estimate holds in $u\in \mathcal H_K^p$.
Therefore, $u$ converges to zero in $L^p(\partial K,d\mu_K)$ as $r\rightarrow0^+$.

On the other hand, every element $\Phi(\sigma)e^{-\frac{\eta r^{p'}}{p'}}$ in $\mathcal M_{p,K,H}$ has radial trace $\Phi$ as $r\rightarrow0^+$.
Thus zero is the only element in $\mathcal M_{p,K,H}$
that belongs to $\mathcal H_K^p$ when $p>N$.
In this range $p>N$, the results in Theorems~\ref{t2} and~\ref{tHcompletion}
remain valid, as their proofs require only $L^p$ distances,
although the nonzero functions in
$\mathcal M_{p,K,H}$ do not belong to $\mathcal H_K^p$.
\end{remark}

\section*{Acknowledgments}
The authors would like to thank Professors Nguyen Lam and Deping Ye
for their valuable guidance, helpful comments, and support.
This work was supported by
Natural Science Basic Research Program of Shaanxi Province of China (2026JC-YBMS-0065).

\bigbreak

\noindent Zhenzhen Wei\\
Department of Mathematics and Statistics,
Memorial University of Newfoundland,
St. John's, Newfoundland and Labrador A1C 5S7, Canada\\
E-mail: {\tt zhenzhenw@mun.ca}\\

\noindent Denghui Wu\\
College of Science, Northwest A\&F University,
Yangling, Shaanxi 712100, China\\
E-mail: {\tt wudenghui66@163.com}\\

\end{document}